\documentclass{article}
\usepackage{fullpage}
\usepackage{color}
\usepackage[normalem]{ulem}
\usepackage{amsmath,amsthm,amssymb,enumerate}
\usepackage{euscript,mathrsfs}
\usepackage[colorlinks,backref]{hyperref}
\usepackage[left=2cm,right=2cm,top=3.5cm,bottom=3.5cm]{geometry}
\usepackage{color}
\catcode`\@=11 \@addtoreset{equation}{section}

\catcode`\@=12

\usepackage{amsfonts}
\usepackage{color}
\usepackage{framed}
\usepackage{bm}
\usepackage{mathtools}
\usepackage{cite}
\usepackage{pgf}
\usepackage{lmodern}
\usepackage{pifont}
\usepackage{framed}
\usepackage{enumitem}
\usepackage{stmaryrd}
\usepackage{amssymb}

\allowdisplaybreaks

\newtheorem{Theorem}{Theorem}[section]
\newtheorem{Proposition}[Theorem]{Proposition}
\newtheorem{Lemma}[Theorem]{Lemma}
\newtheorem{Corollary}[Theorem]{Corollary}

\theoremstyle{definition}
\newtheorem{Definition}[Theorem]{Definition}

\newtheorem{Remark}[Theorem]{Remark}

\newcommand{\bTheorem}[1]{
\begin{Theorem} \label{T#1} }
\newcommand{\eT}{\end{Theorem}}

\newcommand{\bProposition}[1]{
\begin{Proposition} \label{P#1}}
\newcommand{\eP}{\end{Proposition}}

\newcommand{\bLemma}[1]{
\begin{Lemma} \label{L#1} }
\newcommand{\eL}{\end{Lemma}}

\newcommand{\bCorollary}[1]{
\begin{Corollary} \label{C#1} }
\newcommand{\eC}{\end{Corollary}}

\newcommand{\bRemark}[1]{
\begin{Remark} \label{R#1} }
\newcommand{\eR}{\end{Remark}}

\newcommand{\bDefinition}[1]{
\begin{Definition} \label{D#1} }
\newcommand{\eD}{\end{Definition}}

\newcommand{\dif}{\mathrm{d}}
\newcommand{\Dif}{\mathrm{d}}

\newcommand{\mf}{\mathbb{F}}
\newcommand{\mr}{\mathbb{R}}
\newcommand{\prst}{\mathbb{P}}
\newcommand{\p}{\mathbb{P}}

\newcommand{\mn}{\mathbb{N}}

\newcommand{\mt}{\mathbb{T}^3}

\newcommand{\bu}{\mathbf u}

\newcommand{\StoB}{\left(\Omega, \mf,(\mf_t )_{t \geq 0},  \mathbb{P}\right)}

\newcommand{\bfv}{\mathbf{v}}

\newcommand{\bfr}{\mathbf{r}}
\newcommand{\bfe}{\mathbf{e}}

\newcommand{\bfF}{\mathbf{F}}

\newcommand{\bFormula}[1]{
\begin{equation} \label{#1}}
\newcommand{\eF}{\end{equation}}

\newcommand{\Ov}[1]{\overline{#1}}

\newcommand{\vr}{\varrho}

\newcommand{\vu}{\vc{u}}
\newcommand{\vm}{\vc{m}}

\newcommand{\vc}[1]{{\bf #1}}

\newcommand{\Div}{{\rm div}}
\newcommand{\Grad}{\nabla}
\newcommand{\Dt}{\frac{\rm d}{{\rm d}t}}

\newcommand{\dx}{\,{\rm d} {x}}

\newcommand{\dt}{\,{\rm d} t }

\newcommand{\vU}{\vc{U}}
\newcommand{\vV}{\vc{V}}

\newcommand{\intO}[1]{\int_{\T^3} #1 \dx}

\newcommand{\vv}{\vc{v}}

\newcommand{\D}{{\rm d}}
\newcommand{\ep}{h}

\newcommand{\R}{\mathbb{R}}

\newcommand{\T}{\mathbb{T}}
\newcommand{\N}{\mathbb{N}}

\newcommand{\E}{\mathbb{E}}
\newcommand{\intTor}[1]{\int_{\T^3} #1 \ dx}

\def\softd{{\leavevmode\setbox1=\hbox{d}%
          \hbox to 1.05\wd1{d\kern-0.4ex{\char039}\hss}}}

\newcommand{\defeq}{\vcentcolon=}

\newcommand{\dih}[1]{\left(\textnormal{div}_h\,#1\right)_K}
\newcommand{\diht}[1]{\left(\widetilde{\textnormal{div}_h}\,#1\right)_K}
\newcommand{\laph}[1]{\left(\Delta_h\,#1\right)_K}
\newcommand{\laphs}[1]{\left(\Delta^s_h\,#1\right)_K}

\newcommand{\grad}{\nabla\, }

\newcommand{\bfQ}{\mathbf{Q}}
\newcommand{\bff}{\mathbf{f}}

\newcommand{\avrgK}[1]{(\widetilde{#1})^p_{K}}
\newcommand{\avrg}[1]{(\Ov{#1})_{\sigma}}

\newcommand{\jump}[1]{\left\llbracket#1\right\rrbracket}

\newcommand{\n}[2]{\left\lVert{#2}\right\rVert_{#1}}

\newcommand{\grid}{\mathcal{T}}
\newcommand{\elem}{K}
\newcommand{\elemL}{L}

\newcommand{\edge}{\mathrm{E}}

\newcommand{\bv}[1]{{#1}_{\sigma}}

\newcommand{\dv}[1]{{#1}_{K}}

\newcommand{\fluxder}[1]{\left(\partial_h^p {#1}\right)_K}
\newcommand{\der}[1]{\left(\widetilde{\partial_h^p} {#1}\right)_{K}}
\newcommand{\derp}[1]{\left(\partial_h^{p+} {#1}\right)_{K}}
\newcommand{\derm}[1]{\left(\partial_h^{p-} {#1}\right)_{K}}

\newcommand{\faces}{\mathrm{E}}

\newcommand{\facesK}{\faces(K)}

\newcommand{\vrh}{\vr_K}

\newcommand{\bfu}{\mathbf{u}}

\DeclareMathOperator{\divv}{div}

\newcommand{\intOT}[1]{\int_{0}^T {#1}\dt }

\newcommand{\intOhi}[1]{\int_{K}{#1}\dx}

\definecolor{Cgrey}{rgb}{0.85,0.85,0.85}
\definecolor{Cblue}{rgb}{0.50,0.85,0.85}
\definecolor{Cred}{rgb}{1,0,0}
\definecolor{fancy}{rgb}{0.10,0.85,0.10}

\newcommand\Cbox[2]{%
    \newbox\contentbox%
    \newbox\bkgdbox%
    \setbox\contentbox\hbox to \hsize{%
        \vtop{
            \kern\columnsep
            \hbox to \hsize{%
                \kern\columnsep%
                \advance\hsize by -2\columnsep%
                \setlength{\textwidth}{\hsize}%
                \vbox{
                    \parskip=\baselineskip
                    \parindent=0bp
                    #2
                }%
                \kern\columnsep%
            }%
            \kern\columnsep%
        }%
    }%
    \setbox\bkgdbox\vbox{
        \color{#1}
        \hrule width  \wd\contentbox %
               height \ht\contentbox %
               depth  \dp\contentbox
        \color{black}
    }%
    \wd\bkgdbox=0bp%
    \vbox{\hbox to \hsize{\box\bkgdbox\box\contentbox}}%
    \vskip\baselineskip%
}

\date{}

\begin{document}


\title{A convergent finite volume scheme for the stochastic barotropic \\compressible Euler equations}

\author{Abhishek Chaudhary \footnotemark[2] 
\and Ujjwal Koley \footnotemark[3]
}

\date{\today}

\maketitle

\medskip
\centerline{$^\dagger$ Centre for Applicable Mathematics, Tata Institute of Fundamental Research}
\centerline{P.O. Box 6503, GKVK Post Office, Bangalore 560065, India}\centerline{abhi@tifrbng.res.in}

\medskip
\centerline{$\ddagger$ Centre for Applicable Mathematics, Tata Institute of Fundamental Research}
\centerline{P.O. Box 6503, GKVK Post Office, Bangalore 560065, India}
\centerline{ ujjwal@math.tifrbng.res.in}

\medskip

\begin{abstract}
In this paper, we analyze a semi-discrete finite volume scheme for the three-dimensional barotropic compressible Euler equations driven by a \emph{multiplicative} Brownian noise. We derive necessary 
a priori estimates for numerical approximations, and show that the Young measure generated by the numerical approximations converge to a \emph{dissipative measure--valued martingale} solution to the stochastic compressible Euler system. These solutions are probabilistically weak in the sense that the driving noise and associated filtration are integral part of the solution. Moreover, we demonstrate \emph{strong} convergence of numerical solutions to the regular solution of the limit systems at least on the lifespan of the latter, thanks to the weak (measure-valued)--strong uniqueness principle for the underlying system. To the best of our knowledge, this is the first attempt to prove the convergence of numerical approximations for the underlying system.
\end{abstract}


\tableofcontents

\section{Introduction}
Most real world models involve a large number of parameters and coefficients which cannot be exactly determind. Furthermore, there is a considerable uncertainty in the source terms, initial or boundary data due to empirical approximations or measuring errors. Therefore, study of PDEs with randomness (stochastic PDEs) certainly leads to greater understanding of the actual physical phenomenon.
In this paper, we are interested in a stochastic variant of the \emph{compressible barotropic
Euler system}, a set of balance laws driven by
a \emph{nonlinear multiplicative} noise for mass density $\vr$ and the bulk velocity $\vu$ describing the flow of isentropic gas, where the thermal effects are neglected. The system of equations read
\begin{equation}
\begin{aligned}
 \label{P1}
\D \vr + \Div (\vr \vu) \dt &= 0, \\ 
\D (\vr \vu) + \left[ \Div (\vr \vu \otimes \vu)+a  \Grad \vr^\gamma \right]  \dt  &=  \Psi(\vr, \vr \vu) \,\D W
\end{aligned}
\end{equation}
Here $\gamma>1$ denotes the adiabatic exponent, $a>0$ is the squared reciprocal of the Mach number (the ratio between average velocity and speed of sound). The driving process $W$ is a cylindrical Wiener process defined on some filtered probability space $(\Omega,\mf,\{\mf_t\}_{t\ge0},\p)$, and the noise coefficient $\Psi$ is nonlinear and satisfies suitable growth assumptions (see Subsection~\ref{stoc} for the complete list of assumptions). Note that $(\vr, \vr \vu) \mapsto \Psi(\vr, \vr \vu)$ is a given Hilbert space valued function signifying the \emph{multiplicative} nature of the noise. We consider the stochastic compressible Euler equations \eqref{P1}--\eqref{P2} in three spatial dimensions on a periodic domain i.e., on the torus $\mathbb{T}^3$.
The initial conditions are random variables
\begin{equation} \label{P2}
\vr(0,\cdot) = \vr_0, \ \vr \vu(0, \cdot) = (\vr \vu)_0,
\end{equation}
with sufficient spatial regularity to be specified later. 

\subsection{Compressible Euler Equations}
The deterministic counterpart of the stochastic compressible Euler equations \eqref{P1}--\eqref{P2} have received considerable attention and, in spite of monumental efforts, satisfactory well-posedness results are still lacking. It is well-known that the smooth solutions to deterministic counterpart of \eqref{P1}--\eqref{P2} exists only for a finite lap of time, after which singularities may develop for a generic class of initial data. Therefore, global-in-time (weak) solutions must be sought in the class of discontinuous functions. But, weak solutions may not be uniquely determind by their initial data and admissibility conditions must be imposed to single out the physically correct solution. However, the specification of such an admissibility criteria is still open. Indeed, thanks to recent phenomenal work by De Lellis $\&$ Szekelyhidi \cite{DelSze3,DelSze}, and further investigated by Chiodaroli et. al. \cite{chio01}, Feireisl \cite{Feireisl002}, it is well understood that the compressible Euler equations is desparetly ill-posed, due to the lack of compactness of functions satisfying the equations. Even if the initial data is smooth, the global existence and uniqueness of solutions can fail. Moreover, a quest for the existence of global-in-time weak solutions to deterministic counterpart of \eqref{P1}--\eqref{P2} for \emph{general} initial data remains elusive. Given this status quo, it is natural to seek an alternative solution paradigm for compressible Euler system. To that context, we recall the 
framework of \emph{dissipative} Young measure-valued solutions in the context of compressible Navier--Stokes system, being first introduced by Neustupa in \cite{Neus}, and subsequently revisited by Feireisl et. al. in \cite{Fei01}. In a nutshell, these solutions are characterized by a parametrized Young measure and a concentration Young measure in the total energy balance, and they are defined globally in time. 

The study of stochastic compressible Euler equations \eqref{P1}--\eqref{P2}
is a relatively new area of focus within the broarder field of stochastic PDEs, and a satisfactory well/ill-posedness result is largely out of reach. However, we want to emphasize that, to design efficient numerical schemes it is of paramount importance to have prior knowledge about the existence of global-in-time solutions for the underlying system of equations. Without such knowledge, there is no way to establish whether or not the solution produced by a numerical scheme is an approximation of the true solution. To that context, let us first mention the work by Berthelin $\&$ Vovelle \cite{BV13}, where the authors established the existence of a martingale solution for \eqref{P1}--\eqref{P2} in \emph{one} spatial dimension. Moreover, a recent work by Breit et. al. in \cite{BrFeHo2017} revealed that ill-posedness issues for compressible Euler system driven by \emph{additive} noise, in the sense of \cite{DelSze3, Feireisl002}, persist even in the presense of a random forcing. We mention that for compressible Euler equations driven by \emph{multiplicative} noise, the existence of \emph{dissipative measure-valued martingale} solutions was very recently established by Hofmanova et. al. in \cite{MKS01} (see also \cite{K1} for the incompressible case). The authors have shown that the existence can be obtained from a sequence of solutions of stochastic Navier Stokes equations using tools from martingale theory and Young measure theory. 

\subsection{Numerical Schemes}
Parallel to mathematical efforts there has been a huge effort to derive effective numerical schemes for deterministic fluid flow equations, and there is a considerable body of literature dealing with the convergence of numerical schemes for the specific problems in fluid mechanics represented through the barotropic Euler system. In this context, we first mention the work by Karper in \cite{karper} where he has established the convergence of a mixed finite element-discontinuous Galerkin scheme to compressible Euler system under the assumption $\gamma>3$. Subsequently, a series of works \cite{FL17,FLM,FLM1} by Feireisl and his collaborators analyzed the convergence issues for several different semi-discrete numerical schemes via the framework of dissipative measure-valued solutions. Note that the concept of measure–valued solutions introduced in Feireisl et. al. \cite{FLM} (and also \cite{MKS01}) requires the solutions generated by approximate sequences satisfying only the general energy bounds. This is very different from many classical approach where the
existence of measure-valued solution is conditioned by mostly rather unrealistic assumptions of boundedness of certain
physical quantities and the corresponding fluxes. Indeed, assuming
only uniform lower bound on the density and uniform upper bound on the energy they showed that the Lax-Friedrichs-type finite volume schemes generate the dissipative measure–valued solutions
to the barotropic Euler equations.
We also mention that the first numerical evidence that indicated ill-posedness of the Euler system was presented by Elling \cite{ell}. Finally, we mention a series of recent works by Fjordholm et. al. \cite{ulrik_01,ulrik_02} in the context of a general system of hyperbolic conservation laws, where they proved the convergence of a semi-discrete entropy stable finite volume scheme to the \emph{measure-valued} solutions under certain appropriate assumptions.

We remark that, despite the growing interest about the theory of stochastic PDEs and the discretization of stochastic PDEs, the specific question about numerical approximations of stochastic compressible Euler equations is virtually untouched. In fact, the challenges related to numerical aspects of \eqref{P1} are manifold and mostly open, due to the presence of \emph{multiplicative} noise term in \eqref{P1}. Having said this, we mention that there are few results available on stochastic \emph{incompressible} Euler equations. To that context, concerning the convergence of the numerical methods, we mention the work of Brze\'zniak et. al. \cite{carelli}, where the scheme is based on finite elements combined with implicit Euler method.

\subsection{Scope and Outline of the Paper}

The above discussions clearly highlight the lack of effective convergent numerical schemes, for compressible fluid flow equations driven by a \emph{multiplicative} Brownian noise, which are able to take the inherent uncertainties into account, and are equipped with modules that quantify the level of uncertainty. The challenges related to numerical aspects of the underlying problems are mostly open and the research on this frontier is still in its infancy. In fact, the main objective of this article is to lay down the foundation for a comprehensive theory related to numerical methods for \eqref{P1}--\eqref{P2}. Although our work bears some similarities with recent wroks of Fjordholm et. al \cite{ulrik_01,ulrik_02} on deteministic system of conservation laws, and works of Feireisl et. al \cite{FL17,FLM,FLM1} on deterministic Euler systems, the main novelty of this work lies in successfully handling the \emph{multiplicative} noise term. Our problems need to invoke ideas from numerical methods for SDE and meaningfully fuse them with available approximation methods for deterministic problems. This is easier said than done as any such attempt has to capture the noise-noise interaction as well. In the realm of stochastic conservation laws, noise-noise interaction terms play a fundamental role to establish well-posedness theory, for details see \cite{BhKoleyVa, BhKoleyVa_01, BisKoleyMaj, Koley1, Koley2,koley2013multilevel,Koley3}. 

The main contributions of this paper are listed below:
\begin{itemize}
\item [(1)] We develop an appropriate mathematical framework of \emph{dissipative measure-valued martingale} solutions to the stochastic compressible Euler system, keeping in mind that this framework would allow us to establish weak (measure-valued)--strong uniqueness principle. We remark that our solution framework requires only natural energy bounds associated to approximate solutions.
\item [(2)] We show that a Lax-Friedrichs-type numerical scheme for \eqref{P1}--\eqref{P2} generates the \emph{dissipative measure-valued martingale} solutions to the stochastic compressible Euler equations. With the help of the new framework based on the theory of measure--valued solutions, we adapt the concept of $\mathcal{K}$-convergence, first developed in the context of Young measures by Balder \cite{Balder} (see also Feireisl et. al. \cite{FLM1}), to show the pointwise convergence of arithmetic averages (Cesaro means) of numerical solutions to a \emph{dissipative measure-valued martingale} solution of the limit system \eqref{P1}--\eqref{P2}.
\item [(3)] When solutions of the limit continuous problem possess maximal regularity, by making use of weak (measure-valued)--strong uniqueness principle, we show \emph{unconditional} strong $L^1$-convergence of numerical approximations to the regular solution of the limit systems.
\end{itemize}

A breif description of the organization of the rest of the paper is as follows:  we describe all necessary mathematical/technical framework and state the main results in Section~\ref{E}. Moreover, we introduce a Lax-Friedrichs-type finite volume numerical scheme for the underlying system \eqref{P1}--\eqref{P2}. Section~\ref{stability} is devoted on deriving stability properties of the scheme, while Section~\ref{consistency} is focused on deriving suitable formulations of the continuity and momentum equations, and exhibit consistency. In Section~\ref{proof1}, we present a proof of convergence of numerical solutions to a dissipative measure-valued martingale solutions using stochastic compactness.  Section~\ref{w-s} is devoted on deriving the weak (measure-valued) – strong uniqueness principle by making use of a suitable relative energy inequality. Section~\ref{dissipative solution 1} uses the concept of $\mathcal{K}$-convergence to exhibit the pointwise convergence of numerical solutions. Finally, in Section~\ref{proof2}, we make use of weak (measure-valued)--strong uniqueness property to show the convergence of numerical approximations to the solutions of stochastic compressible Euler system \eqref{P1}--\eqref{P2}.


\section{Preliminaries and Main Results}
\label{E}

Here we first briefly recall some relevant mathematical tools which are used in the subsequent analysis and then we state main results of this paper. To begin, we fix an arbitrary large time horizon $T>0$. For the sake of simplicity it will be assumed $a=1$, since its value is not relevant in the present setting. Throughout this paper, we use the letter $C$ to denote various generic constants that may change from line to line along the proofs. Explicit tracking of the constants could be possible but it is highly cumbersome and avoided for the sake of the reader. Let $\mathcal{M}_b(E)$ denote the space of bounded Borel measures on $E$ whose norm is given by the total variation of measures. It is the dual space to the space of continuous functions vanishing at infinity $C_0(E)$ equipped with the supremum norm. Moreover, let $\mathcal{P}(E)$ be the space of probability measures on $E$.

\subsection{Analytic framework}
Let $\gamma \in (0,1)$ be given, and $Z$ be a separable Hilbert space. Let $W^{\gamma,2}(0,T;Z)$ denotes a $Z$-valued Sobolev space which is characterized by its norm
$$
\| g\|^2_{W^{\gamma,2}(0,T;Z)}:= \int_0^T \| g(t)\|^2_{Z}\,dt + \int_0^T \int_0^T \frac{\| g(t)-g(s)\|^2_Z}{|t-s|^{1+2\gamma}}\,dt\,ds.
$$
Then we have following compact embedding result from Flandoli $\&$ Gatarek \cite[Theorem 2.2]{FlandoliGatarek}.
\begin{Lemma}
\label{comp}
If $Z \subset\subset Y$ are two Banach spaces with compact embedding, and real numbers $\gamma \in(0,1)$ satisfy $\gamma >1/2$, then the following embedding
$$
W^{\gamma,2}(0,T;Z) \subset\subset C([0,T]; Y)
$$
is compact.
\end{Lemma}

\subsubsection{Young measures, concentration defect measures}
\label{ym}
In this subsection, we first briefly recall the notion of Young measures and related results which have been used frequently in the text. For an excellent overview of applications of the Young measure theory to hyperbolic conservation laws, we refer to Balder \cite{Balder}. Let us begin by assuming that $(Z, \mathrm{M}, \mu)$ is a sigma finite measure space. A Young measure from $Z$ into $\R^M$ is a weakly measurable function $\mathcal{V}: Z \rightarrow \mathcal{P}(\R^M)$ in the sense that $x \rightarrow \mathcal{V}_{x}(A)$ is $\mathrm{M}$-measurable for every Borel set $A$ in $\R^M$. In what follows, we make use of the following generalization of the classical result on Young measures; for details, see \cite[Section~2.8]{BrFeHobook}.
\begin{Lemma}\label{lem001}
Let $N,M\in\mathbb{N}$, $\mathcal{Q} \subset \R^N\times (0,T) $ and let $({\bf W}_n)_{n \in \N}$, ${\bf W}_n: \Omega\times\mathcal{Q}  \to \R^M$,  be a sequence of random variables such that
$$
\E \big[ \|{\bf W}_n \|^p_{L^p(\mathcal{Q})}\big] \le C, \,\, \text{for a certain}\,\, p\in(1,\infty).
$$
Then on the standard probability space $\big([0,1], \overline{\mathcal{B}[0,1]}, \mathcal{L} \big)$, there exists a new subsequence $(\widetilde {\bf W}_n)_{n \in \N}$ (not relabeled), and a parametrized family ${\lbrace \mathcal{\widetilde V}^{\omega}_{y} \rbrace}_{y \in \mathcal{Q}}$ (superscript $\omega$ emphasises the dependence on $\omega$) of random probability measures on $\R^M$, regarded as a random variable taking values in $\big(L_{w^*}^{\infty}(\mathcal{Q}; \mathcal{P}(\R^M)), w^* \big)$, such that ${\bf W}_n$ has the same law as $\widetilde {\bf W}_n$, i.e. $ {\bf W}_n \sim_{d} \widetilde {\bf W}_n,$
and  the following property holds: for any Carath\'eodory function $J=J(y, Z), y \in \mathcal{Q}, Z \in \R^M$, such that
$$
|J(y,Z)| \le C(1 + |Z|^q), \quad 1 \le q < p, \,\,\text{uniformly in}\,\,y,
$$
implies $\mathcal{L}$-a.s.,
$$
J(\cdot, \widetilde {\bf W}_n) \rightharpoonup \overline{J}\,\,\text{in}\,\, L^{p/q}(\mathcal{Q}), 
\,\, \text{where}\,\, \overline{J}(y) = \langle \mathcal{\widetilde V}^{\omega}_{(\cdot)}; J(y,\cdot)\rangle := \int_{\R^M} J(y, z)\,\D \mathcal{\widetilde V}^{\omega}_{y}(z), \,\,\text{for a.a.}\,\, y \in \mathcal{Q}.
$$
\end{Lemma}
\noindent In literature, Young measure theory has been successfully exploited to extract limits of bounded continuous functions. However, for our purpose, we need to deal with typical functions $F$ for which we only know that 
\begin{equation*}
\E \big[ \|F({\bf W}_n)\|^p_{L^1(\mathcal{Q})}\big] \le C, \,\, \text{for a certain}\,\, p\in(1,\infty), \, \mbox{uniformly in } n.
\end{equation*}
In fact, using a well-known fact that $L^1(\mathcal{Q})$ is embedded in the space of bounded Radon measures $\mathcal{M}_b(\mathcal{Q})$, we can infer that $\p$-a.s.
\begin{align*}
\mbox{weak-* limit in} \, \mathcal{M}_b(\mathcal{Q})\,\, \mbox{of} \, \,F({\bf W}_n) = \langle \mathcal{\widetilde V}^{\omega}_{y}; F\rangle \,dy+ F_{\infty},
\end{align*}
where $F_{\infty} \in \mathcal{M}_b(\mathcal{Q})$, and $F_{\infty}$ is called \emph{concentration defect measure (or concentration Young measure)}. We remark that, a simple truncation analysis and Fatou's lemma reveal that $\p$-a.s. $\| \langle \mathcal{\widetilde V}^{\omega}_{(\cdot)}; F\rangle\|_{L^1(\mathcal{Q})} \leq C$ and thus $\p$-a.s. $\langle \mathcal{\widetilde V}^{\omega}_{y}; F \rangle $ is finite for a.e. $y\in \mathcal{Q}$. In what follows, regarding the concentration defect measure, we shall make use of the following crucial lemma. For a proof of the lemma modulo cosmetic changes, we refer to Feireisl et. al \cite[Lemma 2.1]{Fei01}. 

\begin{Lemma}
\label{lemma001}
Let $\{{\bf W}_n\}_{n > 0}$, ${\bf W}_n: \Omega \times \mathcal{Q} \rightarrow \mathbb{R}^M$ be a sequence generating a Young measure $\{\mathcal{V}^{\omega}_y\}_{y\in \mathcal{Q}}$, where $\mathcal{Q}$ is a measurable set in $\mathbb{R}^N \times (0,T)$. Let $G: \mathbb{R}^M \rightarrow [0,\infty)$
be a continuous function such that
\begin{equation*}
\sup_{n >0} \E\big[\|G({\bf W}_n)\|^p_{L^1(\mathcal{Q})} \big]< \infty, \, \text{for a certain}\,\, p\in(1,\infty),
\end{equation*}
and let $F$ be continuous such that
\begin{equation*}
F: \mathbb{R}^M \rightarrow \mathbb{R}, \quad |F(\bm{z})|\leq G(\bm{z}), \mbox{ for all } \bm{z}\in \mathbb{R}^M.
\end{equation*}
Let us denote $\p$-a.s.
\begin{equation*}
{F_{\infty}} := {\widetilde{F}}- \langle \mathcal{\widetilde V}^{\omega}_{y}, F(\textbf{v}) \rangle \,dy, \quad 
{G_{\infty}} := {\widetilde{G}}- \langle \mathcal{\widetilde V}^{\omega}_{y}, G(\textbf{v}) \rangle \,dy.
\end{equation*}
Here ${\widetilde{F}}, {\widetilde{G}} \in \mathcal{M}_b(\mathcal{Q})$ are weak-$*$ limits of $\{F({\bf W}^n)\}_{n > 0}$, $\{G({\bf W}^n)\}_{n > 0}$ respectively in $\mathcal{M}_b(\mathcal{Q})$. Then $\p$-almost surely $|F_{\infty}| \leq G_{\infty}$.
\end{Lemma}

\subsubsection{Convergence of arithmetic averages}
Following Feireisl et. al. \cite{FLM1}, we also show that the
arithmetic averages of numerical solutions converge pointwise to a generalized dissipative solution of the compressible Euler system, as introduced in Hofmanova et. al. \cite{MKS01}. To that context, we have the following result.

\begin{Proposition}\label{kconvergence}
\label{prop1}
Let $(X,\mathcal{A},\mu)$ be a finite measure space, and $\vU_n\rightharpoonup\vU$ weakly in $L^1(X;\R^M)$. Then there exists a subsequence $(\vU_{n_k})_{k\ge\,1}$ of sequence $(\vU_n)_{n\,\ge1}$ such that
\begin{align*}
\frac {1} {n} \sum_{k=1}^n \vU_{n_k} \to \vU, \quad \mbox{  a.e. in } {X}.
\end{align*}
\end{Proposition}
\begin{proof}
Since the sequence $(\vU_n)_{n\,\ge\,1}$ is uniformly bounded in $L^1(X)$, thanks to Koml\'os theorem, there exists a subsequence $(\vU_{n_k})_{k\,\ge\,1}$ and $\tilde{\vU}\in L^1(X)$ such that
\begin{align*}
\frac {1} {n} \sum_{k=1}^n \vU_{n_k} \to \tilde{\vU}, \quad \mbox{  a.e. in } {X}.
\end{align*}

Let us define $\vV_n:=\frac {1} {n} \sum_{k=1}^n \vU_{n_k}$. Since $\vU_{n_k}$ is also converges weakly to $\vU$, it implies that $\vV_n$ converges weakly to $\vU$ in $L^1(X)$. So sequence ${V_n}$ is uniformly integrable in $L^1(X)$. As consequence of Vitali's convergence theorem implies that $\vV_n$ converges to $\tilde\vU$ strongly in $L^1(X)$. Therefore, uniqueness of weak limit implies that $\vU=\tilde\vU$ in $L^1(X)$. This concludes the proof.
\end{proof}

\subsection{Background on Stochastic framework}
\label{stoc}
Here we briefly recapitulate some basics of stochastic calculus in order to 
define the cylindrical Wiener process $W$ and the stochastic integral appearing in \eqref{P1}. To that context, let $(\Omega,\mf,(\mf_t)_{t\geq0},\prst)$ be a stochastic basis with a complete, right-continuous filtration. The stochastic process $W$ is a cylindrical $(\mf_t)$-Wiener process in a separable Hilbert space $\mathfrak{W}$. It is formally given by the expansion
$$W(t)=\sum_{k\geq 1} e_k W_k(t),$$
where $\{ W_k \}_{k \geq 1}$ is a sequence of mutually independent real-valued Brownian motions relative to $(\mf_t)_{t\geq0}$ and $\{e_k\}_{k\geq 1}$ is an orthonormal basis of $\mathfrak{W}$.
To give the precise definition of the diffusion coefficient $\Psi$, consider $\varrho\in L^\gamma(\mt)$, $\varrho\geq0$, and $\bfu\in L^2(\mt)$ such that $\sqrt\varrho\bfu\in L^2(\mt)$. 
Denote $\bf m=\varrho\bf u$ and let $\,\Psi(\varrho,{\bf m}):\mathfrak{W}\rightarrow L^1(\mt)$ be defined as follows
$$\Psi(\varrho,{\bf m})e_k=\Psi_k(\cdot,\varrho(\cdot),{\bf m}(\cdot)).$$
The coefficients $\Psi_{k}:\mt\times\mr\times\mr^3\rightarrow\mr^3$ are $C^1$-functions that satisfy uniformly in $x\in\mt$
\begin{align}
\Psi_k (\cdot, 0 , 0) &= 0 \label{FG1}\\
| \partial_\vr \Psi_k | + |\nabla_{{\bf m}} \Psi_k | &\leq \beta_k, \quad \sum_{k \geq 1} \beta_k  < \infty.
\label{FG2}
\end{align}
As usual, we understand the stochastic integral as a process in the Hilbert space $W^{-m,2}(\mt)$, $m>3/2$. Indeed, it is easy to check that under the above assumptions on $\varrho$ and $\bf m$, the mapping $\Psi(\varrho,\varrho\bf u)$ belongs to $L_2(\mathfrak{W};W^{-m,2}(\mt))$, the space of Hilbert--Schmidt operators from $\mathfrak{W}$ to $W^{-m,2}(\mt)$.
Consequently, if\footnote{Here $\mathcal{P}$ denotes the predictable $\sigma$-algebra associated to $(\mf_t)$.}
\begin{align*}
\varrho&\in L^\gamma(\Omega\times(0,T),\mathcal{P},\dif\prst\otimes\dif t;L^\gamma(\mt)),\\
\sqrt\varrho\bfu&\in L^2(\Omega\times(0,T),\mathcal{P},\dif\prst\otimes\dif t;L^2(\mt)),
\end{align*}
and the mean value $(\varrho(t))_{\mt}$ is essentially bounded
then the stochastic integral
\[
\int_0^t \Psi(\vr, \vr \vu) \ {\rm d} W = \sum_{k \geq 1}\int_0^t \Psi_k (\cdot, \vr, \vr \vu) \ {\rm d} W_k
\]
is a well-defined $(\mf_t)$-martingale taking values in $W^{-m,2}(\mt)$. Note that the continuity equation \eqref{P1} implies that the mean value $(\varrho(t))_{\mt}$ of the density $\varrho$ is constant in time (but in general depends on $\omega$).
Finally, we define the auxiliary space $\mathfrak{W}_0\supset \mathfrak{W}$ via
$$\mathfrak{W}_0:=\bigg\{u=\sum_{k\geq1}\beta_k e_k;\;\sum_{k\geq1}\frac{\beta_k^2}{k^2}<\infty\bigg\},$$
endowed with the norm
$$\|u\|^2_{\mathfrak{W}_0}=\sum_{k\geq1}\frac{\beta_k^2}{k^2},\quad v=\sum_{k\geq1}\beta_k e_k.$$
Note that the embedding $\mathfrak{W}\hookrightarrow \mathfrak{W}_0$ is Hilbert--Schmidt. Moreover, trajectories of $W$ are $\prst$-a.s. in $C([0,T];\mathfrak{W}_0)$.

For the convergence of approximate solutions, it is necessary to secure strong compactness (a.s. convergence) in the $\omega$-variable. For that purpose, we need a version of Skorokhod representation theorem, so-called Skorokhod-Jakubowski representations theorem. Note that classical Skorokhod theorem only works for Polish spaces, but in our analysis path spaces are so-called quasi-Polish spaces. In this paper, we use the following version of the Skorokhod-Jakubowski theorem, taken from Brze\'zniak et.al. \cite{BrzezniakHausenblasRazafimandimby}.

\begin{Theorem}\label{skorokhod}
Let $\mathcal{X}$ be a complete separable metric space and $\mathcal{Y}$
be a topological space such that there is a sequence of continuous functions
$g_n : \mathcal{Y} \rightarrow \R$ that separates points of $\mathcal{Y}$. Let $(\Omega,\mf,(\mf_t)_{t\geq0},\prst)$ be a stochastic basis with a complete, right-continuous filtration and $(\xi_n)_{n \in \N}$ be a tight sequence of random variables in $(\mathcal{Z}, \mathcal{B}(\mathcal{X}) \otimes \mathcal{M})$, where $\mathcal{Z} = \mathcal{X} \times \mathcal{Y}$ and $\mathcal{Z}$ is equipped with the topology induced by the canonical projections $\Pi_1 : \mathcal{Z} \rightarrow \mathcal{X}$ and $\Pi_2 : \mathcal{Z} \rightarrow \mathcal{Y}$. Note that $\mathcal{M}$ is the $\sigma$-algebra generated by the sequence $\xi_n$, ${n \in \N}$.

Assume that there exists a random variable $\eta$ in $\mathcal{X}$ such that 
$\prst^{\Pi_1 \circ \xi_n} = \prst^{\eta}$. Then there exists a subsequence $(\xi_{n_k})_{k \in \N}$ and random variables $\tilde \xi_k, \tilde \xi$ in $\mathcal{Z}$ for $k \in \N$ on a common probability space $(\tilde \Omega, \tilde \mf, \tilde \prst)$ with
\begin{itemize}
\item [(a)] ${\tilde \prst}^{\tilde \xi_k} = \prst^{\xi_{n_k}}$
\item [(b)] $\tilde \xi_k \rightarrow  \tilde \xi$ in $\mathcal{Z}$ almost surely for $k \rightarrow \infty$.
\item [(c)] $\Pi_1 \circ \tilde \xi_k = \Pi_1 \circ \tilde \xi$ almost surely.
\end{itemize}
\end{Theorem}
%

Finally, we mention the ``Kolmogorov test''  for the existence of continuous modifications of real-valued stochastic processes.

\begin{Lemma}
\label{lemma01}
Let $X={\lbrace X(t) \rbrace}_{t \in [0,T]} $ be a real-valued stochastic process defined on a probability space $(\Omega,\mf,(\mf_t)_{t\geq0},\prst)$. Suppose that there are constants $a >1, b >0$, and $C>0$ such that for all $s,t \in [0,T]$,
\begin{align*}
\E[|X(t)-X(s)|^a] \le C |t-s|^{1 + b}.
\end{align*}
Then there exists a continuous modification of $X$ and the paths of $X$ are $c$-H\"{o}lder continuous for every $c \in [0, \frac{b}{a})$.
\end{Lemma}

\subsection{Stochastic compressible Euler equations}

Since we aim at proving pointwise convergence of numerical solutions to the regular solution of the limit system, using the weak (measure-valued)--strong uniqueness principle for dissipative measure-valued solutions, we first recall the notion of local strong pathwise solution for stochastic compressible Euler equations, being first introduced in \cite{BrMe}. Such a solution is strong in both the probabilistic and PDE sense, at least locally in time. To be more precise, system \eqref{P1}--\eqref{P2} will be satisfied pointwise (not only in the sense of distributions) on the given stochastic basis associated to the cylindrical Wiener process $W$.

\begin{Definition}[Local strong pathwise solution] \label{def:strsol}
	
Let $\StoB$ be a stochastic basis with a complete right-continuous filtration. Let ${W}$ be an $(\mf_t) $-cylindrical Wiener process and $(\varrho_0,\bfv_0)$ be a $W^{m,2}(\T^3)\times W^{m,2}(\T^3)$-valued $\mf_0$-measurable random variable, for some $m>7/2$, and let $\Psi$ satisfy \eqref{FG1} and \eqref{FG2}.
A triplet
$(\varrho,\vv,\mathfrak{t})$ is called a local strong pathwise solution to the system \eqref{P1}--\eqref{P2} provided
\begin{enumerate}
\item $\mathfrak{t}$ is an a.s. strictly positive  $(\mf_t)$-stopping time;
\item the density $\varrho$ is a $W^{m,2}(\mt)$-valued $(\mf_t)$-progressively measurable process satisfying
$$\varrho(\cdot\wedge \mathfrak{t})  > 0,\ \varrho(\cdot\wedge \mathfrak{t}) \in C([0,T]; W^{m,2}(\mt)) \quad \mathbb{P}\text{-a.s.};$$
\item the velocity $\vv$ is a $W^{m,2}(\mt)$-valued $(\mf_t)$-progressively measurable process satisfying
$$ \vv(\cdot\wedge \mathfrak{t}) \in   C([0,T]; W^{m,2}(\mt))\quad \mathbb{P}\text{-a.s.};$$
\item  there holds $\prst$-a.s.
\[
\begin{split}
\varrho (t\wedge \mathfrak{t}) &= \varrho_0 -  \int_0^{t \wedge \mathfrak{t}} \Div(\varrho\vv ) \ \dif s, \\
(\varrho \vv) (t \wedge \mathfrak{t})  &= \varrho_0 \vv_0 - \int_0^{t \wedge \mathfrak{t}} \Div (\varrho\vv \otimes\vv ) \ \dif s 
- \int_0^{t \wedge \mathfrak{t}}a\Grad \varrho^\gamma \ \dif s + \int_0^{t \wedge \mathfrak{t}} {\Psi}(\varrho,\varrho\vv ) \ \D W,
\end{split}
\]
for all $t\in[0,T]$.
\end{enumerate}
\end{Definition}

Note that classical solutions require spatial derivatives of $\vv$ and $\vr$ to be continuous $\prst$-a.s. This motivates the following definition.

\begin{Definition}[Maximal strong pathwise solution]\label{def:maxsol}
	Fix a stochastic basis with a cylindrical Wiener process and an initial condition as in Definition \ref{def:strsol}. A quadruplet $$(\varrho,\vv,(\mathfrak{t}_R)_{R\in\mn},\mathfrak{t})$$ is a maximal strong pathwise solution to system \eqref{P1}--\eqref{P2} provided
	
	\begin{enumerate}
		\item $\mathfrak{t}$ is an a.s. strictly positive $(\mf_t)$-stopping time;
		\item $(\mathfrak{t}_R)_{R\in\mn}$ is an increasing sequence of $(\mf_t)$-stopping times such that
		$\mathfrak{t}_R<\mathfrak{t}$ on the set $[\mathfrak{t}<T]$,
		$\lim_{R\to\infty}\mathfrak{t}_R=\mathfrak t$ a.s. and
		\begin{equation*}
		\sup_{t\in[0,\mathfrak{t}_R]}\|\vv(t)\|_{1,\infty}\geq R\quad \text{on}\quad [\mathfrak{t}<T] ;
		\end{equation*}
		\item each triplet $(\varrho,\vv,\mathfrak{t}_R)$, $R\in\mn$,  is a local strong pathwise solution in the sense  of Definition \ref{def:strsol}.
	\end{enumerate}
\end{Definition}

There are quite a few results available in the literature concerning the existence of maximal pathwise solutions for various SPDE or SDE models, see for instance \cite{BMS,Elw}. For compressible Euler equations, a specific work can be found in Breit $\&$ Mensah in \cite[Theorem 2.4]{BrMe}.

\begin{Theorem}\label{thm:main}
	Let $m>7/2$ and the coefficients
	$\Psi_k$ satisfy hypotheses \eqref{FG1}, \eqref{FG2} and let $(\varrho_0,\bfv_0)$ be an $\mathbb{F}_0$-measurable, $W^{m,2}(\mt)\times W^{m,2}(\mt)$-valued random variable such that $\varrho_0>0$ $\p$-a.s. Then
	there exists a unique maximal strong pathwise solution, in the sense of Definition \ref{def:maxsol}, $(\varrho,\vv,(\mathfrak{t}_R)_{R\in\mn},\mathfrak{t})$ to problem \eqref{P1}--\eqref{P2}
	with the initial condition $(\varrho_0,\vv_0)$.
\end{Theorem}

\subsection{Measure-valued solutions}

For the introduction of measure-valued solutions, it is convenient to work with the following reformulation of the problem \eqref{P1}--\eqref{P2} in the \emph{conservative} variables $\varrho$ and ${\bf m}= \varrho {\bf u}$:
\begin{align} \label{P11}
\D \vr + \Div \,{\bf m} \dt &= 0,\\ \label{P21}
\D {\bf m} + \left[ \Div \bigg(\frac{{\bf m} \otimes {\bf m}}{\varrho}\bigg)+  \Grad p(\vr) \right]  \dt  &=  \Psi (\vr, {\bf m}) \,\D W.
\end{align}
Note that, in general any uniformly bounded sequence in $L^1(\T^3)$ does not immediately imply weak convergence of it due to the presence of oscillations and concentration effects. To overcome such a problem, two kinds of tools are used:
\begin{itemize}
\item [(a)] Young measures: these are probability measures on the phase space and accounts for the persistence of oscillations in the solution;
\item [(b)] Concentration defect measures: these are measures on physical space-time, accounts for blow up type collapse due to possible concentration points.
\end{itemize}


\subsubsection{Dissipative measure-valued martingale solutions }

Keeping in mind the previous discussion, we now introduce the concept of \textit{dissipative measure--valued martingale solution} to the stochastic compressible Euler system. In what follows, let 
\[
\mathcal{M} = \left\{ [\vr, \vm] \ \Big| \ \vr \geq 0, \ \vm \in \R^3 \right\}
\]
be the phase space associated to the Euler system. 

\begin{Definition}[Dissipative measure-valued martingale solution]
\label{def:dissMartin}
Let $\Lambda$ be a Borel probability measure on $L^\gamma(\mathbb{T}^3)\times L^{\frac{2\gamma}{\gamma+1}}(\mathbb{T}^3)$. Then $\big[ \big(\Omega,\mathbb{F}, (\mathbb{F}_{t})_{t\geq0},\mathbb{P} \big); \mathcal{V}^{\omega}_{t,x}, W \big]$ is a dissipative measure-valued martingale solution of \eqref{P11}--\eqref{P21}, with initial condition $\mathcal{V}^{\omega}_{0,x}$; if
\begin{enumerate}
\item [(01)]$\mathcal{V}^{\omega}$ is a random variable taking values in the space of Young measures on $L^{\infty}_{w^*}\big([0,T] \times \T^3; \mathcal{P}\big(\mathcal{M})\big)$. In other words, $\p$-a.s.
$\mathcal{V}^{\omega}_{t,x}: (t,x) \in [0,T] \times \T^3  \rightarrow \mathcal{P}(\mathcal{M})$ is a parametrized family of probability measures on $\mathcal{M}$,
\item [(02)]$ \big(\Omega,\mathbb{F}, (\mathbb{F}_{t})_{t\geq0},\mathbb{P} \big)$ is a stochastic basis with a complete right-continuous filtration,
\item [(03)]$W$ is a $(\mathbb{F}_{t})$-cylindrical Wiener process,
\item [(04)]the average density $\langle \mathcal{V}^{\omega}_{t,x}; \varrho \rangle$ satisfies $t\mapsto \langle \langle \mathcal{V}^{\omega}_{t,x}; \varrho \rangle(t, \cdot),\varphi\rangle\in C[0,T]$ for any $\varphi\in C^\infty(\mathbb{T}^3)$ $\mathbb{P}$-a.s., the function $t\mapsto \langle \langle \mathcal{V}^{\omega}_{t,x}; \varrho \rangle(t, \cdot),\varphi\rangle$ is progressively measurable and 
\begin{align*}
\mathbb{E}\, \bigg[ \sup_{t\in(0,T)}\Vert  \langle \mathcal{V}^{\omega}_{t,x}; \varrho \rangle(t,\cdot)\Vert_{L^\gamma(\mt)}^p\bigg]<\infty 
\end{align*}
for all $1\leq p<\infty$,
\item [(05)]the average momentum $\langle \mathcal{V}^{\omega}_{t,x}; \textbf{m} \rangle$ satisfies  $t\mapsto \langle \langle \mathcal{V}^{\omega}_{t,x}; \textbf{m} \rangle (t, \cdot),\bm\varphi\rangle\in C[0,T]$ for any $\bm{\varphi}\in C^\infty(\mathbb{T}^3)$ $\mathbb{P}$-a.s., the function $t\mapsto \langle \langle \mathcal{V}^{\omega}_{t,x}; \textbf{m} \rangle (t, \cdot),\bm{\varphi}\rangle$ is progressively measurable and 
\begin{align*}
\mathbb{E}\, \bigg[ \sup_{t\in(0,T)}\Vert  \langle \mathcal{V}^{\omega}_{t,x}; \textbf{m} \rangle (t, \cdot) \Vert_{L^\frac{2\gamma}{\gamma+1}(\mt)}^p\bigg]<\infty 
\end{align*}
for all $1\leq p<\infty$,
\item [(06)]$\Lambda=\mathcal{L}[\mathcal{V}^{\omega}_{0,x}]$,
\item [(07)]the integral identity
\begin{equation} 
\begin{aligned}\label{first condition measure-valued solution}
\int_{\T^3} \langle \mathcal{V}^{\omega}_{\tau,x}; \varrho \rangle \, \varphi  \,dx -\int_{\T^3} \langle \mathcal{V}^{\omega}_{0,x}; \varrho \rangle\, \varphi \,dx 
= \int_{0}^{\tau} \int_{\T^3}  \langle \mathcal{V}^{\omega}_{t,x}; \textbf{m}\rangle \cdot \nabla_x \varphi \,dx \,dt 
\end{aligned}
\end{equation}
holds $\p$-a.s., for all $\tau \in[0,T)$, and for all $\varphi \in C^{\infty}(\T^3)$,
\item [(08)]the integral identity 
\begin{equation} \label{second condition measure-valued solution}
\begin{aligned}
&\int_{\T^3} \langle \mathcal{V}^{\omega}_{\tau,x};\textbf{m}\rangle \cdot \bm{\varphi}dx - \int_{\T^3} \langle \mathcal{V}^{\omega}_{0,x};\textbf{m}\rangle \cdot \bm{\varphi}dx \\
&\qquad = \int_{0}^{\tau} \int_{\T^3} \left[\left\langle \mathcal{V}^{\omega}_{t,x}; \frac{\textbf{m}\otimes \textbf{m}}{\varrho} \right\rangle: \nabla_x \bm{\varphi} + \langle \mathcal{V}^{\omega}_{t,x};p(\varrho)\rangle \divv_x \bm{\varphi} \right] dx dt \\
&\qquad \qquad+ \int_{\T^3} \bm{\varphi}\,\int_0^{\tau} \left\langle \mathcal{V}^{\omega}_{t,x}; {\Psi}(\varrho,\textbf{m} )\right\rangle \, \D W \,dx+ \int_{0}^{\tau} \int_{\T^3}  \nabla_x \bm{\varphi}: d\mu_m,
\end{aligned}
\end{equation}
holds $\p$-a.s., for all $\tau \in [0,T)$, and for all $\bm{\varphi} \in C^{\infty}(\T^3;\mathbb{R}^3)$, where $\mu_m\in L^{\infty}_{w^*}\big([0,T]; \mathcal{M}_b({\T^3} )\big)$, $\p$-a.s., is a tensor--valued measure,
\item [(09)]there exists a real-valued martingale $M_{E}$, such that the following energy inequality
\begin{align} \label{third condition measure-valued solution}
\begin{aligned}
\mathrm{E}(t+) & \leq  \mathrm{E}(s-) \\
&\quad + \frac{1}{2} \int_s^{t} \bigg(\int_{\T^3} \sum_{k = 1}^{\infty} \left\langle \mathcal{V}^{\omega}_{\tau,x};\varrho^{-1}| \Psi_k (\varrho, \textbf{m}) |^2 \right\rangle \bigg)\, d\tau + \frac12\int_s^{t} \int_{\T^3} d \mu_e + \int_s^{t}  dM_{E} 
\end{aligned}			
\end{align}
holds $\p$-a.s., for all $0 \le s <t$ in $(0,T)$ with
\begin{align*}
&\mathrm{E}(t-):=\lim_{\tau \rightarrow 0+} \frac{1}{\tau} \int_{t -\tau}^t \Bigg(\int_{\T^3} \left\langle \mathcal{V}^{\omega}_{s,x}; \frac{1}{2} \frac{|\textbf{m}|^2}{\varrho} +P(\varrho) \right \rangle \,dx + \mathcal{D}(s)\Bigg) \,ds\\
&\mathrm{E}(t+):=\lim_{\tau \rightarrow 0+} \frac{1}{\tau} \int_t^{t +\tau} \Bigg(\int_{\T^3} \left\langle \mathcal{V}^{\omega}_{s,x}; \frac{1}{2} \frac{|\textbf{m}|^2}{\varrho} +P(\varrho) \right \rangle \,dx + \mathcal{D}(s)\Bigg) \,ds 
\end{align*}
Here $P(\varrho):= \frac{\varrho^{\gamma}}{\gamma-1}$, $\mu_e\in L^{\infty}_{w^*}\big([0,T]; \mathcal{M}_b({\T^3} )\big)$, $\p$-a.s., $\mathcal{D}\in L^{\infty}(0,T)$, $\mathcal{D}\geq 0$, $\p$-a.s., with initial energy. 
$$
\mathrm{E}(0-)= 
\int_{\mathbb{T}^3} \bigg(\frac{1}{2} \frac{|\textbf{m}_0|^2}{\varrho_0} +P(\varrho_0) \bigg)\,dx.
$$
\item [(10)]there exists a constant $C>0$ such that
\begin{equation} \label{fourth condition measure-valued solutions}
\int_{0}^{\tau} \int_{\T^3} d|\mu_m| + \int_{0}^{\tau} \int_{\T^3} d|\mu_e| \leq C \int_{0}^{\tau} \mathcal{D}(t) dt,
\end{equation}	
holds $\p$-a.s., for every $\tau \in (0,T)$.
\end{enumerate}
\end{Definition}

\begin{Remark}
We remark that, in light of a standard Lebesgue point argument applied to \eqref{third condition measure-valued solution}, energy inequality holds for a.e. $0 \le s <t$ in $(0,T)$: 
\begin{equation}
\begin{aligned}
\label{energy_001}
&  \int_{\T^3} \left\langle \mathcal{V}^{\omega}_{t,x}; \frac{1}{2} \frac{|\textbf{m}|^2}{\varrho} +P(\varrho) \right \rangle \,dx + \mathcal{D}(t)\\
&\qquad \leq  \int_{\T^3}  \left\langle \mathcal{V}^{\omega}_{s,x}; \frac{1}{2} \frac{|\textbf{m}|^2}{\varrho} +P(\varrho)\right \rangle  \,dx + \mathcal{D}(s) + \frac{1}{2} \int_s^{t} \bigg(\int_{\T^3} \sum_{k = 1}^{\infty} \left\langle \mathcal{V}^{\omega}_{\tau,x};\varrho^{-1}| \Psi_k (\varrho, \textbf{m}) |^2 \right\rangle \bigg)\, d\tau \\
& \qquad \qquad+ \frac12\int_s^{t} \int_{\T^3} d \mu_e + \int_s^{t}  dM^2_{E}, \,\,\p-a.s.
\end{aligned}
\end{equation}
However, to establish weak (measure-valued)--strong uniqueness principle, we require energy inequality to hold for \emph{all} $s, t \in (0,T)$. This can be achieved following the argument depicted in Section~\ref{proof1}.
\end{Remark}

\begin{Remark}
Note that the above solution concept slightly differs from the dissipative measure-valued martingale solution concept introduced by Martina et. al. \cite{MKS01}. Indeed, the main difference lies in the successful identification of the martingale term present in \eqref{second condition measure-valued solution}.
\end{Remark}

\subsection{Numerical scheme}
\label{scheme}

It is well known that standard finite difference, finite volume and finite element methods have been very successful in computing solutions to system of hyperbolic conservation laws, including deterministic compressible fluid flow equations. Here we consider a semi-discrete finite volume scheme for the stochastic compressible Euler equations \eqref{P1}--\eqref{P2}. In what follows, drawing preliminary motivation from the analysis depicted in \cite{FL17,FLM,FLM1}, we describe the finite volume numerical scheme which is later shown to converge in appropriate sense. More precisely, we show that  the sequence of numerical solutions generate the Young measure that represents the dissipative measure-valued martingale solution.

\subsubsection{Spatial discretization}
We begin by introducing some notation needed to define the semi-discrete finite volume scheme. Throughout this paper, we reserve the parameter $h$ to denote small positive numbers that represent the spatial discretizations parameter of the numerical scheme.
Note that, since we are working in a periodic domain in $\R^3$, the relevant domain for the space discretization is $[0,\ell]^3,$ $\ell>0$. To this end, we introduce the space 
discretization by finite volumes (control volumes). For that we need to recall the definition of so called  admissible meshes for finite 
volume scheme.

\begin{Definition}[Admissible mesh] \label{def:admissible mesh}
	An admissible mesh $\mathcal{T}$ of $[0,\ell]^3$  is
	a family of disjoint regular quadrilateral connected subset of $[0,\ell]^3$ satisfying the following:
	\begin{itemize}
		\item [i)] $[0,\ell]^3$ is the union of the closure of the elements (called control volume K) of $\mathcal{T}$, i.e., $[0,\ell]^3:= \cup_{K \in \mathcal{T}} \bar{K}$.
		\item[ii)] The common interface of any
		two elements of $\mathcal{T}$  is included in a hyperplane of $[0,\ell]^3$.
		\item[iii)] There exists nonnegative constant $\alpha$ such that 
		\begin{align*}
		\begin{cases}
		\alpha h^3 \le |K|, \\
		|\partial K | \le \frac{1}{\alpha}h^{2}, \quad \forall K\in \mathcal{T},
		\end{cases}
		\end{align*}
		where $h= \sup\big\{ \text{diam} (K): K\in \mathcal{T}\big\} < + \infty$, $|K|$ denotes the $3$-dimensional Lebesgue measure of $K$, and 
		$|\partial K |$ represents the $2$-dimensional Lebesgue measure of $\partial K $.
	\end{itemize}
\end{Definition}
In the sequel, we denote the followings:
\begin{itemize}
	\item  $\mathrm{E}_K$: the set of interfaces of the control volume $K$.
	\item $\mathcal{N}(K)$: the set of control volumes neighbors of the control volume $K$.
	\item $\sigma_{K,L}$: the common interface between $K$ and $L$, for any $L\in \mathcal{N}(K)$.
	\item $\mathrm{E}$: the set of all the interfaces of the mesh $\mathcal{T}$.
	\item $\mathbf{n}_K^+$: the unit normal vector to interface $\sigma_{K,L}$, oriented from $K$ to $L$, for any $ L \in \mathcal{N}(K)$.
	\item $\bfe_p$: the unit basis vector in the $p$-th space direction, $p =1, 2, 3.$ Note that in our case the mesh is a regular  quadrilateral grid, and thus $\mathbf{n}_K^{+}$ is parallel to $\mathbf{e}_p$, for some $p=1,2,3.$
\end{itemize}

Let $ \mathcal{Y}(\grid)$ denote the space of piecewise constant functions defined on admissible mesh $\grid.$ For $w_h \in \mathcal{Y}(\grid)$ we set
$\displaystyle w_K \equiv w_{h_{|_{\elem}}}.$  Then it holds that
\begin{align*}
\intO{w_h}=h^3\sum_{K\in\grid} w_K .
\end{align*}
The value of  $W_h $ on the face $\sigma$ shall be denoted by $\bv{W},$ and analogously for faces $s\pm$ of cell $K$ in $\pm\bfe_s$ direction.
We also introduce a standard projection operator 
\begin{align*}
\Pi_h : L^1(\T^3) \rightarrow \mathcal{Y}(\grid), \quad (\Pi_h(\varphi))_{K}\defeq\frac{1}{h^3}\int_{\elem}{\varphi(x)\,dx}.
\end{align*}
For $w_h, W_h \in X(\grid)$ we define the following discrete operators
\begin{align*}
\der{w_h}&\defeq \frac{w_L-w_J}{2h}, \ \derp{w_h}\defeq \frac{w_L-w_K}{h},\ \derm{w_h}\defeq \frac{w_K-w_J}{h},
\quad L=K+h\bfe_p, J=K-h\bfe_p, \\
\fluxder{W_h}&\defeq \frac{W_{\sigma,p+}-W_{\sigma,p-}}{h}, \quad p =1, 2, 3.
\end{align*}
The discrete  Laplace and divergence operators
are defined as follows
\begin{align*}
\laph{w_h}&\defeq \frac{1}{h^2}\sum_{L\in\mathcal{N}(K)}(w_L-w_K)=\sum_{p=1}^3\laphs{w_h}, \\  \diht{{\bf w}_h}&\defeq \sum_{p=1}^3 \der{w_h^p}, \quad
\dih{{\bf W}_h} \defeq \sum_{p=1}^3 \fluxder{W_h^p}.
\end{align*}
Furthermore, on the face $\sigma=K|L \in \edge$ we define the  jump and mean value operators
\begin{align*}
\jump{w_h}:=w_L \bold{n}_K^+ + w_K\bold{n}_K^-, \quad \avrg{w_h} \defeq  \frac{w_K+w_L}{2}, \quad L=K+h\bfe_p,\ p=1,2,3,
\end{align*}
respectively. Here $\mathbf{n}_K^+,$ $\mathbf{n}_K^- \equiv \mathbf{n}_L^+$ denote the unit  outer normal to $\elem$ and $\elemL,$ respectively. Finally, we introduce the mean value of $w_h \in \mathcal{Y}(\grid)$ in cell $\elem$ in the direction of $\bfe_p$ by
\begin{align*}
\avrgK{w_h}:=\frac{w_L+w_J}{2}, \quad L=K+h\bfe_p, \ J=K-h\bfe_p.
\end{align*}

\subsubsection{Entropy stable flux and the scheme}
\label{sec:scheme_01}
Note that constructing and analyzing numerical schemes for the deterministic counterpart of the underlying system of equations \eqref{P1}--\eqref{P2} has a long tradition. Usually the schemes are developed to satisfy certain additional properties like entropy condition and kinetic energy stability which can be important for
turbulent flows. To that context, Tadmor \cite{Tad03} proposed the idea of entropy conservative numerical fluxes which can then be combined with some dissipation terms using entropy variables to obtain a scheme that respects the entropy condition, i.e., the scheme must
produce entropy in accordance with the second law of thermodynamics. Such a flux is called \emph{entropy stable flux}.

In order to introduce the finite volume numerical scheme for the underlying system of equations, let us first recast the system of equations \eqref{P11}--\eqref{P21} in the following form:
\begin{align*}
\D \vU(t) + \mathrm{div} f(\vU) \,dt &= \mathbb{H}\big(\vr, \vm \big)\,\D W(t),\\
\vU(t,0)&=\vU_0,
\end{align*}
where we introduced the variables $\vU =[\vr, \vm]$, $f(\vU)= [\vm, \frac{\vm \otimes \vm}{\vr} + p(\vr)\mathbb{I}]$, and $\mathbb{H}(\vr, \vm)=[0, \Psi(\vr, \vm)]$. 

We propose the following semi-discrete (in space) finite volume scheme approximating the underlying system of equations \eqref{P11}--\eqref{P21}
\begin{align}\label{entropy_bE}
\D \vU_K(t) + \diht{{\bf F_h}(t)} \,dt &= \mathbb{H}\big(\vr_K(t), \vm_K(t)\big)\,\D W(t), \,\, t>0, \, K \in \grid, \\
 \vU_K(0) &= (\Pi_h (\vU_0))_K, \,\, K \in \grid. \nonumber
\end{align}
Note that \eqref{entropy_bE} is a stochastic differential equation in $V:=L^\gamma(\mathbb{T}^3)\times L^{\frac{2\gamma}{\gamma+1}}(\mathbb{T}^3)$. Let us now specify the numerical flux ${\bf F_h} := {\bf F_h}(\vU_K, \vU_L)$ associated to the flux function $f$. Indeed, we want ${\bf F_h}$ to satisfy the following properties:
\begin{itemize}
	\item [(a)] (Consistency) The function ${\bf F_h}$ satisfies
	$ {\bf F_h}(a,a)=f(a)$, for all $a\in V$.
	\item[(b)] (Lipschitz continuity) There exist two constants $F_1, F_2 >0$ such that  for any $a,b,c \in V$, it holds that
	\begin{align*}
	\big|{\bf F_h}(a,b)-{\bf F_h}(c,b)\big| \le F_1 \big| a-c|, \\
	\big|{\bf F_h}(a,b)-{\bf F_h}(a,c)\big| \le F_2 \big|b-c|.
	\end{align*}
	\item[(c)] (Entropy stability) The flux ${\bf F_h}$ is entropy stable.
\end{itemize}

 Note that there are plethora of numerical fluxes available in literature satisfying the above three conditions. However, to illustrate the main ideas, we will consider a scheme with a Lax-Friedrichs-type numerical flux $\bfF_h$ (which is entropy stable) whose value on a face $\sigma = K|L$ is given by
\begin{align}\label{num_flux}
\bv{\bfF}\defeq \avrg{\bff(\vU_h)}-\bv{\lambda} \jump{\vU_h}.
\end{align}
Here the global diffusion coefficient is
$\displaystyle \bv{\lambda} \equiv \lambda\defeq \max_{K\in\grid}\max_{s=1,\ldots,N} |\lambda^s(\vU_K)|$, while the local diffusion coefficient is
$\displaystyle \bv{\lambda} \defeq \max_{s=1,\ldots,N} \max( |\lambda^s (\vU_K)|, |\lambda^s(\vU_L)|).$
Note that $\lambda^s$ is the $s-$th eigenvalue of the corresponding Jacobian matrix $\bff'(\vU_h)$. We mention that we restrict ourselves to the case of constant numerical viscosities. However, one can easily extend the results to local diffusion case, as presented in \cite{FLM}. 
Using the above notation, we introduce the following semi-discrete finite volume scheme to approximate system \eqref{P11}--\eqref{P21}.
The scheme can be written in the standard  per cell finite volume formulation for all $K \in \grid,$
\begin{subequations}\label{scheme_fv}
\begin{align}
&\D \vr _K(t) + \sum_{\sigma \in \facesK} \frac{|\sigma|}{|K|} F_h(\vr_h(t),\vu_h(t))_K =0,
\\
&\D \vm_K(t) + \sum_{\sigma \in \facesK} \frac{|\sigma|}{|K|}
\left({\bf F}_h(\vr_h(t), \vm_h(t))_K  + \Ov{p_K}(t) \vc{n}\right) =\Psi\big(\vrh(t), \vm_K(t)\big)\,\D W(t).
\end{align}
\end{subequations}

\noindent We can rewrite the above scheme \eqref{scheme_fv} in the following explicit form
\begin{subequations}\label{scheme_bE}
	\begin{align}
	\D {\vr_K(t)}&+\diht{\vm_h(t)}- \lambda h \laph{\vr_h(t)}=0, \label{cont_bE} &\\
	\D {\vm_K(t)}&+\diht{\left(\frac{\vm_h(t)\otimes\vm_h(t)}{\vr_h(t)}+p_h(t)\mathbb{I}\right)}-\lambda  h\laph{\vm_h(t)} =\Psi\big(\vrh(t), \vm_K(t)\big)\,\D W(t),\ t >0 \ K \in \grid. \label{mom_bE}&
	\end{align}
\end{subequations}

\noindent {\bf Existence of numerical solutions.}
Note that the set of equations \eqref{scheme_bE} represent a system of stochastic differential equations. The discrete problem \eqref{scheme_bE} admits a unique (probabilistically) strong solution $(\vr_h(t), \vm_h(t))$  for every $t \in (0,T)$. This follows from the classical results on stochastic differential equations, thanks to positivity of the density $\varrho_K$ (cf. Lemma~\ref{LemP}) and Lipschitz continuity of the fluxes and the noise coefficient. For more details, we refer to \cite[Section 4]{FLM}.

\subsection{Statements of main results}

We now state main results of this paper. To begin with, regarding the convergence of solutions of the numerical scheme, we have the following theorem.

\begin{Theorem} 
\label{thm:exist}
Suppose that the approximate solutions $\{\vU_h=(\vr_h(t), \vm_h(t))\}_{h>0}$ be generated by the scheme \eqref{scheme_bE} for the stochastic Euler system. Moreover, assume that 
$$
0 < \underline{\varrho} \le \vr_h \le \tilde{\vr}, \quad |\vm_h| \le \tilde{\vm}, \,\,\mbox{$\p$-a.s. uniformly for}\,\, h\rightarrow 0,
$$
for some positive constants $\underline{\varrho}, \tilde{\varrho}$, and $\tilde{\vm}$. Then $\{\vU_h\}_{h>0}$ generates a dissipative measure-valued martingale solution to the barotropic Euler system in the sense of Definition~\ref{def:dissMartin}.

Next, we make use of the $\mathcal{K}$-convergence in the context of Young measures to conclude the following pointwise convergence of averages of numerical solutions to a dissipative measure-valued martingale solution to \eqref{P11}--\eqref{P21}.
\end{Theorem}
	\begin{Theorem}
		\label{dissipative solution}Suppose that the approximate solutions $\{\vU_h=(\vr_h(t), \vm_h(t))\}_{h>0}$ to \eqref{scheme_bE} for the stochastic Euler system generate a dissipative measure-valued martingale solution $\big[ \big(\Omega,\mathbb{F}, (\mathbb{F}_{t})_{t\geq0},\mathbb{P} \big); \mathcal{V}^{\omega}_{t,x}, W \big]$ in the sense of Definition~\ref{def:dissMartin}. Then following holds true,
		\item[1.] there exists subsequence $\{\mathbf{U}_{h_k}=(\vr_{h_k}(t), \vm_{h_k}(t))\}_{h_k\textgreater\,0}$ such that, $\p$-a.s.
		$$\varrho_{h_k}\to\langle {\mathcal{V}^{\omega}_{t,x}}; \varrho \rangle\,\,\mbox{in}\,\,\,C_w([0,T],L^\gamma(\mathbb{T}^3)),$$
		$$\mathbf{m}_{h_k}\to\langle {\mathcal{V}^{\omega}_{t,x}}; {\textbf m} \rangle\,\,\mbox{in}\,\,\,C_w([0,T],L^{\frac{2\gamma}{\gamma+1}}(\mathbb{T}^3)).$$ 
		\item[2.] $\p$-a.s., there exists subsequece $\{\mathbf{U}_{h_k}=(\vr_{h_k}(t), \vm_{h_k}(t))\}_{h_k\textgreater\,0}$ such that
		\begin{equation*}
			\begin{aligned}
				\frac 1N \sum_{k=1}^N \vr_{h_k} &\to\langle {\mathcal{V}^{\omega}_{t,x}}; \varrho \rangle, \ \mbox{as $N \rightarrow \infty$ a.e. in} \,\,(0,T)\times\T^3, \\
				\frac 1N \sum_{k=1}^N \vm_{h_k} &\to \langle {\mathcal{V}^{\omega}_{t,x}}; {\textbf m} \rangle, \ \mbox{as $N \rightarrow \infty$ a.e. in} \,\,(0,T)\times\T^3.
			\end{aligned}
		\end{equation*}		
\end{Theorem}
\noindent Finally, making use of the weak (measure-valued)--strong uniqueness principle (cf. Theorem~\ref{Weak-Strong Uniqueness_01}), we prove the following result justifying the strong convergence to the regular solution.
\begin{Theorem}
\label{T_ccE}\ \\
Suppose that the approximate solutions $\{\vU_h\}_{h>0}$ to \eqref{scheme_bE} for the stochastic Euler system generate a dissipative measure-valued martingale solution in the sense of Definition~\ref{def:dissMartin}.
In addition, let the Euler equations \eqref{P11}--\eqref{P21} possess the unique strong (continuously differentiable) solution $(\bar{\vU}, (\mathfrak{t}_R)_{R\in\mn}, \mathfrak{t})=([\bar{\vr},\bar{\vm}], (\mathfrak{t}_R)_{R\in\mn}, \mathfrak{t})$, emanating form the initial data \eqref{P2}. Then $\p$-a.s.
\begin{equation*}
\begin{aligned}
\vr_{h}(\cdot\wedge\mathfrak{t}_R) &\to \bar{\vr}(\cdot\wedge\mathfrak{t}_R) \ \mbox{weakly-(*) in} \ L^{\infty}(0,T;L^\gamma(\T^3))
\ \mbox{and strongly in}\ L^1((0,T) \times \T^3), \\
\vm_{h}(\cdot\wedge\mathfrak{t}_R) &\to \bar{\vm}(\cdot\wedge\mathfrak{t}_R) \ \mbox{weakly-(*) in} \ L^{\infty}(0,T;L^{2\gamma/(\gamma+1)}(\T^3)) \ \mbox{and strongly in}\ L^1((0,T) \times \T^3; \R^3)). 
\end{aligned}
\end{equation*}
\end{Theorem}
\begin{Remark}
Note that the results stated in Theorem~\ref{T_ccE} are unconditional provided that:
\begin{itemize}
\item [(1)]the limit system admits a smooth solution.
\item [(2)] numerical approximations generate a dissipative measure-valued martingale solution.
\end{itemize}
\end{Remark}

\section{Stability of the Numerical Scheme}
\label{stability}
We show the stability of the numerical schemes defined in previous Section by deriving   a priori estimates.

\subsection{A priori estimates for the stochastic Euler system}

The approximate solutions resulting from scheme \eqref{scheme_bE} enjoy the following properties:

\begin{itemize}[wide=0pt]
\item[1.] {\bf Conservation of mass}\\

Multiplying the equation of continuity in \eqref{scheme_bE} by $h^3$ for all $K \in \mathcal{T}$, and integrating in time yields the total mass conservation, i.e., $\p$-a.s.
\begin{equation*}
\intO{ \vr_h (t, \cdot) } = \intO{ \vr_{h}^0 },\ t \geq 0.
\end{equation*}
\item[2.] {\bf Conditional positivity of numerical density} \\

We show positivity of the density under an additional hypothesis on the approximate velocity. We assume that $\p$-a.s.
\begin{equation} \label{vHYPP}
\vc{u}_h \equiv \frac{\vc{m}_h(t)}{\vr_h(t)} \in L^2(0,T; L^\infty(\Omega)).
\end{equation}
	
Thus the first two equations of the numerical scheme for the Euler system  read,
\begin{subequations}\label{aux_scheme_bE}
\begin{align}
\Dt{\vr_K(t)}&+\diht{(\vr_h(t)\vu_h(t))}-\lambda  h\laph{\vr_h(t)}=0, \label{aux_cont_bE} \\
\Dt{(\vr_K(t)\vu_K(t))}&+\diht{\big(\vr_h(t)(\vu_h(t)\otimes\vu_h(t))+p_h(t)\mathbb{I}\big)}-\lambda  h\laph{(\vr_h(t)\vu_h(t))} \notag \\
&\qquad \qquad =\Psi\big(\vrh(t), \vr_K(t)\vu_K(t) \big)\,\D W(t), \label{aux_mom_bE}
\end{align}
\end{subequations}
equipped with the relevant initial conditions.

\begin{Lemma}\label{LemP}
Let $\vr_h(0)>0,$ and let the numerical solution $(\vr_h(t),\vu_h(t)),$ $t>0$ satisfy the discrete continuity equation \eqref{aux_cont_bE}, where we assume $\vu_h$ satisfies \eqref{vHYPP}. Then $\p$-a.s.
\begin{align*}
\dv{\vr}(t)  > \underline{\vr}>  0, \quad t \in [0,T],\  K \in \grid.
\end{align*}
\end{Lemma}
\begin{proof}
To establish the proof, one can follow \cite{FLM} modulo cosmetic changes. The details are left to the interested reader.
\end{proof}
Note that, under the hypothesis (\ref{vHYPP}), setting $\vm_h\equiv \vr_h\vu_h$ and comparing \eqref{aux_cont_bE} with \eqref{cont_bE}, we conclude that both formulations are equivalent.

\item[3.] {\bf Energy estimates} \\

First observe that the positivity of $\vr_h(t)$ implies that $\p$-almost surely $\vr_h \in L^{\infty}(0,T;L^1(\Omega))$. Next, we show that the underlying entropy stable finite volume scheme \eqref{scheme_bE} produces the discrete entropy inequality. To see this, let us denote by
\begin{align*}
\eta(\dv{\vU})= \frac{1}{2}\frac{|\dv{\vm}|^2}{\dv{\vr}} + P(\dv{\vr}),
\end{align*}
where $\vU_K(t) $ solves the equation \eqref{entropy_bE}. 
Now applying It\^o formula to the function $\eta(\vU_K(t))$ and using entropy stability properties of numerical flux functions \cite[Example 5.2]{Tad03}, we get the discrete energy inequality
\begin{align}\label{dis_en_bE}
\D \dv{\eta(\vU}(t))&+\dih{\bfQ_h(t)}\,dt  \\
&\le\,\sum_{k=1}^\infty\Psi_k (\vr_K(t), \vm_K(t))\cdot \vu_K(t)\, \D W_k(t)
+ \frac{1}{2}\sum_{k = 1}^{\infty}  \vr_K(t)^{-1}| \Psi_k (\vr_K(t), \vm_K(t)) |^2 \, {\rm d}t, \notag
\end{align}
where $\bfQ_h$ is a entropy stable flux. Since the numerical entropy flux is conservative, i.e., $\displaystyle\sum_{K\in\grid}\dih{\bfQ_h}=0,$  the integral of \eqref{dis_en_bE}  yields $\mathbb P$-a.s.
\begin{align}\label{001}
\int_{\T^3}{\eta({\bf U}_h}(t))dx
&\le\,\int_{\T^3}{\eta(\vU}_h(0))dx+\int_0^t\int_{\T^3}\sum_{k=1}^\infty\Psi_k (\vr_h(s), \vm_h(s))\cdot \vu_h(s)\,  dx\, \D W_k(s)\\&\qquad
+ \frac{1}{2}\int_0^t\int_{\T^3}\sum_{k = 1}^{\infty}  \vr_h(s)^{-1}| \Psi_k (\vr_h(s), \vm_h(s)) |^2 \,{\rm d}x \,{\rm d}s. \notag
\end{align}
We can apply the $p$-th power on both sides of \eqref{001}, and then take expectation to obtain usual energy bounds. In particular, we have following uniform bounds
\begin{align}
\frac{{\bf m}_h}{\sqrt{ \varrho_h}} &\in L^{p}(\Omega;L^\infty(0,T;L^2( \mathbb T^3))),\label{aprhov}\\
\varrho_h&\in L^{p}(\Omega;L^\infty(0,T;L^\gamma( \mathbb T^3))),\label{aprho}\\
{\bf m}_h &\in L^{p}(\Omega;L^\infty(0,T;L^\frac{2\gamma}{\gamma+1}( \mathbb T^3))),\label{estrhou2}
\end{align}
\begin{Remark}
Note that above estimates are natural in the context of stochastic compressible Euler equations.
\end{Remark}
Let  $\psi \in C_c^{\infty}([0,T))$, $\psi \ge 0$. Applying It\^o product formula to the function $\eta(\vU_K(t))\psi(t)$, we have energy inequality
\begin{align*}
\begin{aligned}
-\int_0^T \partial_t \psi \int_{\mathbb T^3} &\bigg[ \frac{1}{2}\frac{|{\bf m}_{h}|^2}{\varrho_{h}}  + \frac{\varrho^{\gamma}_{h}}{\gamma-1} \bigg] \,dx \,ds \leq \psi(0) \int_{\mathbb T^3} \bigg[ \frac{1}{2}\frac{|{\bf m}_{h}(0)|^2}{\varrho_{h}(0)}  + \frac{\varrho^{\gamma}_{h}(0)}{\gamma-1}\bigg] \dx \\
&\qquad +\sum_{k=1}^\infty\int_0^T \psi \bigg(\int_{\mathbb T^3}\Psi_k (\varrho_{h}, {\bf m}_{h})\cdot{\bf u}_{h}\dx\bigg){\rm d} W_k + \frac{1}{2}\sum_{k = 1}^{\infty}  \int_0^T \psi
\int_{\mathbb T^3} \varrho_{h}^{-1}| \Psi_k (\varrho_{h}, {\bf m}_{h}) |^2 \,dx\, {\rm d}s.
\end{aligned}
\end{align*}
holds $\mathbb P$-a.s., for all $\psi \in C_c^{\infty}([0,T))$, $\psi \ge 0$.

\item[4.] {\bf Additional estimates} \\

Regarding the regularity estimates for the discrete numerical solution, we have the following lemma.
\begin{Lemma}\label{LemP1}
The following relevant estimates with $\Gamma=\frac{2\gamma}{\gamma+1}$ hold $\p$-almost surely:
\begin{align*}
\vr_h \in L^{\infty}(0,T; L^{\gamma}(\mathbb T^3)),  &\quad \vm_h \in L^{\infty}(0,T; L^{\Gamma}(\mathbb T^3)), \quad  
\widetilde{\textnormal{div}_h} \,\vm_h \in L^{\infty}(0,T; W^{-1,\Gamma}(\mathbb T^3)), \\
&\Delta_h \vr_h \in L^{\infty}(0,T; W^{-2, \gamma}(\mathbb T^3)). \nonumber
\end{align*}
\end{Lemma}
\begin{proof}
Note that first two estimates are direct consequences of discrete energy bounds. Next, we note that for any test function $\varphi \in W^{1,\Gamma'}(\mathbb T^3)$
\begin{align*}
\left< \widetilde{\textnormal{div}_h} \,\vm_h(t), \varphi \right> &= h^3 {\sum_{K\in\grid} \diht{\vm_h(t)}\dv{(\Pi_h\varphi)}} \\
&=-{\sum_{K\in\grid} \sum_{s=1}^3 m_K^s(t)\left(\intOhi{\frac{\varphi(x+h\bfe_s) -\varphi(x-h\bfe_s)}{2h}}\right)}\\
&=-{\sum_{K\in\grid} \sum_{s=1}^3 m_K^s(t)\left(\frac{1}{2h}\intOhi{\int_{x+h\bfe_s}^{x-h\bfe_s} \varphi_x(\theta)\,d\theta}\right)} \le \frac{1}{2h} {\sum_{K\in\grid} \sum_{s=1}^3 m_K^s(t)\left(\intOhi{h\varphi_x}\right)} \\
&\le \frac{h^{1+1/\Gamma}}{2h}{\sum_{K\in\grid} \sum_{s=1}^3 m_K^s(t)\left(\intOhi{\varphi^{\Gamma'}_x\,dx}\right)^{\frac{1}{\Gamma'}}}
\le C(\Gamma)\,\|\vm_h\|_{\gamma} \|\varphi\|_{W^{1,\Gamma'}(\mathbb T^3)}.
\end{align*}
This confirms the third estimate. A similar argument yields the result for the discrete Laplacian.
\end{proof}
\end{itemize}


\section{Consistency of the Numerical Scheme}
\label{consistency}
In this section we show consistency of the entropy stable finite volume scheme. In addition, we also exhibit consistency of the energy inequality.

\subsection{Consistency formulation of continuity and momentum equations}

We begin by multiplying the continuity equation \eqref{cont_bE} by $h^3\dv{(\Pi_h\varphi)},$ with $\varphi \in C^3(\mathbb T^3),$ and the momentum equation or \eqref{mom_bE} by $h^3\dv{(\Pi_h\boldsymbol{\varphi})},$ with  $\boldsymbol{\varphi} \in C^3(\T^3;\R^3)$. Then we sum the resulting equations over $K\in\grid$ and integrate in time. For time derivatives in the continuity and momentum equations, it is straightforward to observe that
\begin{align*}
h^3\intOT{\sum_{K\in\grid} \D{\dv{\vr}(t)} \,\dv{(\Pi_h\varphi)}}&=\intOT{\D{} \bigg(\intO{\vr_h(t)\varphi(x)} \bigg)} = \left< \vr_h(T), \varphi \right> - \left< \vr_h(0), \varphi \right>, \\
h^3\intOT{\sum_{K\in\grid} \D {\dv{\vm}(t)}\cdot\dv{(\Pi_h\boldsymbol{\varphi})}} &=\intOT{\D{} \bigg( \intO{\vm_h(t)\cdot \boldsymbol{\varphi}(x)}\bigg)} 
=\left< \vm_h(T), \boldsymbol{\varphi} \right> - \left< \vm_h(0), \boldsymbol{\varphi} \right>.
\end{align*}

\noindent To handle the convective terms, we shall make use of the discrete integration by parts and the Taylor expansion. For the continuity equation, we have
\begin{align*}
&h^3\intOT{\sum_{K\in\grid} \diht{\vm_h(t)}\dv{(\Pi_h\varphi)}}& \\
&=-\intOT{\sum_{K\in\grid} \sum_{s=1}^3 m_K^s(t)\left(\intOhi{\frac{\varphi(x+h\bfe_s) -\varphi(x-h\bfe_s)}{2h}}\right)}
=-\intOT{\intO{\vm_h(t)\cdot \grad\varphi(x)}}+\mathcal{R}_1(h,\varphi),&
\end{align*}
where
term $\mathcal{R}_1(h,\varphi)$ is estimated as follows
\begin{align}\label{r1}
\mathcal{R}_1(h,\varphi)\le C(\varphi) h \n{L_t^{\infty}L^1_x}{\vm_h}, \, \p \,\,\mbox{a.s.}
\end{align}
Similarly, for the convective term in the momentum equations, we have
\begin{align*}
&h^3\intOT{\sum_{K\in\grid} \diht{\bigg(\frac{\vm_h(t)\otimes\vm_h(t)}{\vr_h(t)}+p_h(t)\mathbb{I}\bigg)}\dv{(\Pi_h\boldsymbol{\varphi})}}& \\
&=-\intOT{\sum_{K\in\grid} \sum_{s=1}^3\sum_{z=1}^3\bigg(\frac{m_h^s(t)m_h^z(t)}{\vr_h(t)}+p_h(t)\bigg)\left(\intOhi{\frac{\varphi^z(x+h\bfe_s) -\varphi^z(x-h\bfe_s)}{2h}}\right)}&\\
&=-\intOT{\intO{\bigg(\frac{\vm_h(t)\otimes\vm_h(t)}{\vr_h(t)}+p_h(t)\mathbb{I}\bigg)\cdot\grad\boldsymbol{\varphi}(x)}}+\mathcal{R}_2(h,\bm{\varphi}),&
\end{align*}
where the term $\mathcal{R}_2(h,\bm{\varphi})$ is bounded by
\begin{align*}
\mathcal{R}_2(h,\bm{\varphi})
\le C(\bm{\varphi}) h \Bigg\{\n{L_t^{\infty}L_x^2}{\sqrt{\vr_h(t)}\vu_h(t)}+\n{L_t^{\infty}L^1_x}{p_h(t)}\Bigg\}, \, \p \,\,\mbox{a.s.}
\end{align*}

\noindent Next, regarding the numerical diffusion term with global numerical diffusion coefficients $\lambda$, we have
\begin{align*}
& h^{4}\intOT{ \lambda\sum_{K\in\grid} \laph{\vU_h(t)}\dv{(\Pi_h\boldsymbol{\varphi})}}& \nonumber \\ \nonumber
& = h^{4}\intOT{\lambda\sum_{K\in\grid} \dv{\vU}(t)\left(\intOhi{\sum_{s=1}^N\frac{\boldsymbol{\varphi}(x+h\bfe_s) -2\boldsymbol{\varphi}(x) +\boldsymbol{\varphi}(x-h\bfe_s)}{h^2}}\right)}& \\ 
&= h^3\intOT{\lambda\intO{\vU_h(t)\Delta{\boldsymbol{\varphi}(x)}}} + \mathcal{N}(h,\bm{\varphi}),
\end{align*}
where the term $\mathcal{N}(h,\bm{\varphi})$ is bounded by
\begin{align*}
\mathcal{N}(h,\bm{\varphi}) \le C(\bm{\varphi})   h   \n{L_t^{\infty}L_x^{1}}{\vU_h}\intOT{\lambda}, \, \p \,\,\mbox{a.s.}
\end{align*}
Finally, regarding the stochastic term, we have the following
\begin{align*}
 h^3\int_0^ T\sum_{K\in\grid}  \Psi\big(\vrh(t), \vm_K(t)\big)\,  (\Pi_h \bm{\varphi})_K \,\D W(t) = \int_0^{T}\langle\Psi(\varrho_{h},\mathbf{m}_{h}), \bm{\varphi}\rangle\,\mathrm{d}W(t)
\end{align*}

\noindent Let us summarize the consistency results derived in this section.

\subsubsection*{Consistency formulation for the stochastic Euler system}

The consistency formulation of the numerical schemes \eqref{scheme_bE} for the barotropic Euler equations reads

\begin{enumerate}
\item for all $\varphi\in C^\infty(\mathbb{T}^3)$ and $\bm{\varphi}\in C^\infty(\mathbb{T}^3)$ we have $\mathbb{P}$-a.s. for all $t\in[0,T]$
\begin{align}
\langle \varrho_{h}(t), \varphi\rangle &= \langle\varrho_{h}(0) , \varphi\rangle - \int_0^{t}\langle\mathbf{m}_{h}, \nabla \varphi\rangle\mathrm{d}s + \lambda\,h^3 \int_0^t  \langle \varrho_{h}(t), \Delta \varphi\rangle \,ds+ \mathcal{R}_1(h,\varphi) + \mathcal{N}_1(h,\varphi) \label{eq:energy}
\\
\langle \mathbf{m}_{h}(t), \bm\varphi\rangle &= \langle \mathbf{m}_{h} (0), \bm{\varphi}\rangle - \int_0^{t}\bigg\langle \Big(\frac{\vm_h\otimes\vm_h}{\vr_h}+p(\varrho_{h}) \mathbb{I}\Big), \nabla \bm{\varphi}\bigg\rangle\,\mathrm{d}s + \lambda\,h^3 \int_0^t  \langle \mathbf{m}_{h}(t), \Delta  \bm{\varphi}\rangle \,ds \notag \\
&+\int_0^{t}\langle\Psi(\varrho_{h},\mathbf{m}_{h}), \bm{\varphi}\rangle\,\mathrm{d}W + \mathcal{R}_2(h,\bm{\varphi}) + \mathcal{N}_2(h,\bm{\varphi}). \label{eq:energy1}
\end{align}
\item the energy inequality\index{energy inequality} 
\begin{align}\label{EI2''}
\begin{aligned}
-\int_0^T \partial_t \psi \int_{\mathbb T^3} &\bigg[ \frac{|{\bf m}_{h}|^2}{\varrho_{h}}  + P(\varrho_{h}) \bigg] \,dx \,ds
\leq \psi(0) \int_{\mathbb T^3} \bigg[ \frac{1}{2}\frac{|{\bf m}_{h}(0)|^2}{\varrho_{h}(0)}  + P(\varrho_{h}(0))\bigg] \dx \\
&\qquad +\sum_{k=1}^\infty\int_0^T \psi \bigg(\int_{\mathbb T^3}\Psi_k (\varrho_{h}, {\bf m}_{h})\cdot{\bf u}_{h}\dx\bigg){\rm d} W_k + \frac{1}{2}\sum_{k = 1}^{\infty}  \int_0^T \psi
\int_{\mathbb T^3} \varrho_{h}^{-1}| \Psi_k (\varrho_{h}, {\bf m}_{h}) |^2 \,dx\, {\rm d}s.
\end{aligned}
\end{align}
holds $\mathbb P$-a.s., for all $\psi \in C_c^{\infty}([0,T))$, $\psi \ge 0$.
\end{enumerate}
Here
\begin{equation} \label{error}
\begin{aligned}
\mathcal{R}_1(h,\varphi)&:= \intOT{\sum_{K\in\grid} \sum_{s=1}^3 m_K^s(t)\intOhi{\left(\nabla \varphi -\frac{\varphi(x+h\bfe_s) -\varphi(x-h\bfe_s)}{2h}\right)}}, \\
\mathcal{R}_2(h,\bm{\varphi})&:=\intOT{\sum_{K\in\grid} \sum_{s=1}^3\sum_{z=1}^3\bigg(\frac{m_h^s(t)m_h^z(t)}{\vr_h(t)}+p_h(t)\bigg)\intOhi{\left( \nabla  \bm{\varphi} - \frac{\varphi^z(x+h\bfe_s) -\varphi^z(x-h\bfe_s)}{2h}\right)}}, \\
\mathcal{N}_1(h,\varphi)&:=\lambda\,h\intOT{\sum_{K\in\grid} \dv{\vr}(t)\intOhi{\sum_{s=1}^3 \left(\Delta \varphi- \frac{{\varphi}(x+h\bfe_s) -2{\varphi}(x) +{\varphi}(x-h\bfe_s)}{h^2}\right)}}, \\
\mathcal{N}_2(h,\bm{\varphi})&:=\lambda\,h\intOT{\sum_{K\in\grid}\sum_{s=1}^3 m_K^s(t) \intOhi{\left(\Delta \boldsymbol{\varphi}- \frac{\boldsymbol{\varphi}(x+h\bfe_s) -2\boldsymbol{\varphi}(x) +\boldsymbol{\varphi}(x-h\bfe_s)}{h^2}\right)}}.
\end{aligned}
\end{equation}


\section{Proof of Theorem~\ref{thm:exist}: Existence of Measure-Valued Solution}
\label{proof1}
We shall make use of the given a-priori estimates \eqref{aprhov}, \eqref{aprho}, and \eqref{estrhou2} to pass to the limit in the parameter $h$. In what follows, we begin by the following compactness argument.

\subsection{Compactness and almost sure representations}
\label{subsec:compactness}

Note that, in general, securing a result of compactness in the probability variable ($\omega$-variable) is a non-trivial task. To that context, to obtain strong (a.s.) convergence in the $\omega$-variable, we make use of Skorokhod-Jakubowski's representation theorem (cf. \cite{Jakubowski}). We remark that the classical Skorokhod representation theorem does not work in our setup since our path spaces are not Polish spaces. To overcome this, we use Jakubowski version of Skorokhod representation theorem \cite{Jakubowski} which works for quasi-Polish spaces.

As usual, to establish the tightness of the laws generated by the approximations, we first denote the path space $\mathcal{Y}$ to be the product of the following spaces:
\begin{align*}
\mathcal{Y}_\varrho&=C_w([0,T];L^\gamma(\mt)),&\mathcal{Y}_N&=C([0,T]; \R),\\
\mathcal{Y}_{\textbf{m}}&=C_w([0,T];L^\frac{2\gamma}{\gamma+1}(\mt)),&\mathcal{Y}_W&=C([0,T];\mathfrak{W}_0),\\
\mathcal{Y}_{\mathcal{V}} &= \big(L^{\infty}((0,T)\times \T^3; \mathcal{P}(\R^4)), w^* \big).
\end{align*}
Let us denote by $\mu_{\varrho_h}$, $\mu_{\textbf{m}_h}$, and $\mu_{W_h}$ respectively, the law of $\varrho_h$, $\textbf{m}_h$, and $W_h$ on the corresponding path space. Moreover, let $\mu_{{\mathcal{V}}_h}$, and $\mu_{N_h}$ denote the law of ${\mathcal{V}}_h := \delta_{[\varrho_h, \textbf{m}_h]}$, and $N_h:=\sum_{k\geq1}\int_0^t\intTor{ \vu_{h} \cdot {\Psi_k}(\varrho_h, \textbf{m}_h ) } \,\Dif W$ on the corresponding path spaces. Finally, let $\mu^h$ denotes joint law of all the variables on $\mathcal{Y}$. To proceed further, it is necessary to establish tightness of $\{\mu^h;\,h\in(0,1)\}$. To this end, we observe that tightness of $\mu_{W_\ep}$ is immediate. So we show tightness of other variables.

\begin{Proposition}\label{prop:rhotight}
	The sets $\{\mu_{\varrho_h};\,h\in(0,1)\}$, and $\{\mu_{\textbf{m}_h};\,h\in(0,1)\}$ are tight on path spaces $\mathcal{Y}_\varrho$, $\mathcal{Y}_\bu$, and $\mathcal{Y}_{\textbf{m}}$ respectively. 
\end{Proposition}

\begin{proof}
Proof of this proposition is straightforward, by making use of the a priori bounds given in Lemma~\ref{LemP1}. For the details of this proof, we refer to \cite{BrFeHobook}. 
\end{proof}

\begin{Proposition}\label{rhoutight1}
	The set $\{\mu_{{\mathcal{V}}_{h}};\,h\in(0,1)\}$ is tight on the path space  $\mathcal{Y}_{\mathcal{V}}$.
\end{Proposition}

\begin{proof}
	The aim is to apply the compactness criterion in $\big(L^{\infty}((0,T)\times \T^3; \mathcal{P}(\R^4)), w^* \big)$. Define the set
	\begin{align*}
	B_R:= \Big\lbrace {\mathcal{V}} \in \big(L^{\infty}((0,T)\times \T^3; \mathcal{P}(\R^4)), w^* \big); 
	\int_0^T \int_{\T^3} \int_{\R^4} \Big(|\xi_1|^{\gamma} + |\xi_2|^{\frac{2\gamma}{\gamma+1}} \Big) \,d{\mathcal{V}}_{t,x}(\xi)\,dx\,dt \le R    \Big\rbrace,
	\end{align*}
	which is relatively compact in $\big(L^{\infty}((0,T)\times \T^3; \mathcal{P}(\R^4)), w^* \big)$. Note that
	\begin{align*}
	\mathcal{L}[{\mathcal{V}}_{h}](B^c_R)&=
	\p\Bigg( \int_0^T \int_{\T^3} \int_{\R^4} \Big(|\xi_1|^{\gamma} + |\xi_2|^{\frac{2\gamma}{\gamma+1}} \Big) \,d{\mathcal{V}}_{t,x}(\xi)\,dx\,dt > R  \Bigg) \\
	&= \p\Bigg(\int_0^T \int_{\T^3} \Big(|\varrho_h|^{\gamma} + |\textbf{m}_{h}|^{\frac{2\gamma}{\gamma+1}} \Big) \,dx\,dt >R \Bigg)
	\le \frac1R \E\Big[\|\varrho_h \|_{L^\gamma}^{\gamma} + \| \textbf{m}_{h}\|_{L^\frac{2\gamma}{\gamma+1}}^{\frac{2\gamma}{\gamma+1}}\Big] \le \frac CR.
	\end{align*}
	The proof is complete.
\end{proof}

\begin{Proposition}\label{rhoutight13}
	The set $\{\mu_{N_{h}};\,h\in(0,1)\}$ is tight on the path space $\mathcal{Y}_{N}$.
\end{Proposition}

\begin{proof}
	First observe that, for each $h$, $N_h(t)=\sum_{k\geq1}\int_0^t\intTor{ \vu_{h} \cdot {\Psi_k}(\varrho_h, \textbf{m}_h ) } \,\Dif W$ is a square integrable martingale. Note that for $r > 2$
	\begin{align*}
	\E\Big[ \Big|\sum_{k\ge 1} \int_s^t \int_{\T^3}\vu_{h} \cdot {\Psi_k}(\varrho_h, \textbf{m}_h )\Big|^r  \Big] &\le \E\Big[ \int_s^t \sum_{k=1}^{\infty} \Big|\int_{\T^3}\vu_{h} \cdot {\Psi_k}(\varrho_h, \textbf{m}_h)\Big|^2  \Big]^{r/2} \\
	& \le |t-s|^{r/2}\, \Big(1 + \E \Big[\sup_{0\le t \le T} \| \sqrt{\varrho_h} u_{h}\|^r_{L^2} \Big]\Big)
	\le C|t-s|^{r/2},
	\end{align*}
	and the Kolmogorov continuity criterion i.e., Lemma~\ref{lemma01} applies. This in particular implies that, for some $\alpha>1$
	$$
	\sum_{k\geq1}\int_0^t\intTor{ \vu_{h} \cdot {\Psi_k}(\varrho_h, \textbf{m}_h ) } \,\Dif W \in L^r(\Omega; C^{\alpha}(0,T; \R)).
	$$
	Therefore, tightness of law follows from the compact embedding of $C^{\alpha}$ into $C^0$.
\end{proof}

Combining all the informations obtained from Proposition~\ref{prop:rhotight},  Proposition~\ref{rhoutight1}, and Proposition~\ref{rhoutight13}, we conclude that
\begin{Corollary}
	The set $\{\mu^h;\,h\in(0,1)\}$ is tight on $\mathcal{Y}$. 
\end{Corollary}

At this point, we are ready to apply Jakubowski-Skorokhod representation theorem (see also Brzezniak et. al. \cite{BrzezniakHausenblasRazafimandimby}) to extract a.s convergence on a new probability space. In what follows, passing to a weakly convergent subsequence $\mu^{h}$ (and denoting by $\mu$ the limit law) we infer the following result:

\begin{Proposition}\label{prop:skorokhod1}
There exists a subsequence $\mu^h$ (not relabelled), a probability space $(\widetilde\Omega,\widetilde\mf,\widetilde\prst)$ with $\mathcal{Y}$-valued Borel measurable random variables $(\widetilde\varrho_{h},\widetilde{\bf m}_h, \widetilde W_h, \widetilde N_{h}, {\mathcal{\widetilde V}}_{h})$, $h \in (0,1)$, and  $(\widetilde\varrho,\widetilde{\bf m}, \widetilde W, \widetilde N, {\mathcal{\widetilde V}})$ such that 
\begin{enumerate}
\item [(1)]the law of $(\widetilde\varrho_{h},\widetilde{\bf m}_h, \widetilde W_h, \widetilde N_{h}, {\mathcal{\widetilde V}}_{h})$ is given by $\mu^h$, $h\in(0,1)$,
\item [(2)]the law of $(\widetilde\varrho,\widetilde{\bf m}, \widetilde W,\widetilde N, {\mathcal{\widetilde V}})$, denoted by $\mu$, is a Radon measure,
\item [(3)]$(\widetilde\varrho_{h},\widetilde{\bf m}_h, \widetilde W_h, \widetilde N_{h}, {\mathcal{\widetilde V}}_{h})$ converges $\,\widetilde{\prst}$-almost surely to $(\widetilde\varrho,\widetilde{\bf m},\widetilde{\textbf{u}}, \widetilde W, \widetilde N, {\mathcal{\widetilde V}})$ in the topology of $\mathcal{Y}$, i.e.,
\begin{align*}
&\widetilde\varrho_h \rightarrow \widetilde\varrho \,\, \text{in}\, \,C_w([0,T]; L^{\gamma}(\T^3)), \quad
&\widetilde{\textbf{m}}_h &\rightarrow \widetilde{\textbf m} \,\, \text{in}\, \,C_w([0,T]; L^{\frac{2\gamma}{\gamma+1}}(\T^3)),\\
&\widetilde N_h \rightarrow \widetilde N \,\, \text{in}\, \, C([0,T]; \R), \quad
&\widetilde W_h &\rightarrow \widetilde W \,\, \text{in}\, \,C([0,T]; \mathfrak{W}_0)),\\
& {\mathcal{\widetilde V}}_h  \rightarrow {\mathcal{\widetilde V}}\,\, \text{weak-$*$ in}\, \, L^{\infty}((0,T)\times \T^3; \mathcal{P}(\R^4)).
\end{align*}
\item [(4)] For every $h$, we have $\widetilde W_h = \widetilde W$ $\p$-a.s.
\item [(5)] For Carath\'{e}odory functions $\underline{H}=\underline{H}(t,x,\varrho,\textbf m)$ and $\overline{H}=\overline{H}(t,x,\varrho,\textbf m)$, where $(t,x)\in (0,T)\times \T^3$ and $(\varrho,\textbf m) \in \R^4$, satisfying for some $p, q$ the growth condition
\begin{align*}
|\underline{H}(t,x,\varrho,\textbf m)| & \le 1 + |\varrho|^{p} + |\textbf m|^q,\\
|\overline{H}(t,x,\varrho,\textbf m)| & =\mathcal{O}\Big( 1 + |\varrho|^{\gamma} + |\textbf m|^{\frac{2 \gamma}{\gamma +1}}\Big),
\end{align*}
uniformly in $(t,x)$. Then we have $\widetilde\p$-a.s.
\begin{align*}
& \underline{H}(\widetilde\varrho_h, \widetilde{\textbf{m}}_h) \rightarrow \overline{\underline{H}(\widetilde\varrho, \widetilde{\textbf m})}\,\, \text{in}\,\, L^r((0,T)\times\T^3),\,\, \text{for all}\,\, 1<r\le\frac{\gamma}{p}\wedge \frac{2\gamma}{q(\gamma+1)}, \\
&\overline{H}(\widetilde\varrho_h, \widetilde{\textbf{m}}_h) \rightharpoonup \left\langle \mathcal{\widetilde V}_{(\cdot, \cdot)}; \overline{H}(\widetilde\varrho_h, \widetilde{\textbf{m}}_h) \right\rangle dxdt + \widetilde \mu_{\overline{H}}, \,\, \text{ weak-$*$ in}\, \, L^{\infty}(0,T; \mathcal{M}_b(\T^3)),
\end{align*}
where $\mu_{\overline{H}}$ is the concentration defect measure associated to the function $\overline{H}$. 
\end{enumerate}
\end{Proposition}

\begin{proof}
Proof of the items $(1)$, $(2)$, and $(3)$ directly follow from Jakubowski-Skorokhod representation theorem \cite{Jakubowski}, while item $(4)$ follows from \cite{Jakubowski}, and \cite{BrzezniakHausenblasRazafimandimby}. Finally, item $(5)$ follows from Lemma~\ref{lem001}, and Lemma~\ref{lemma001}.
\end{proof}


\subsubsection{Passage to the limit}
We shall now make use of the above convergences to pass to the limit in approximate equations \eqref{eq:energy}--\eqref{eq:energy1}, and the energy inequality \eqref{EI2''}. To that context, let us first show that the approximations $\widetilde \vr_\ep, \widetilde u_{h}$ solve equations \eqref{eq:energy}--\eqref{eq:energy1} on the new probability space $(\widetilde\Omega,\widetilde\mf,\widetilde\prst)$.
Note that, since $(\varrho_h,\vm_h,N_h)$ are random variables with values in $C([0,T];L^\gamma(\mathbb{T}^3))\times C([0,T];L^{\frac{2\gamma}{\gamma+1}}(\mathbb{T}^3))\times C([0,T];\R)$. By \cite[Lemma A.3]{GA} and \cite[Corollary A.2]{Ondre},  $(\widetilde\varrho_h,\widetilde\vm_h,  \widetilde{N}_h)$ are also random variables with values in $C([0,T],L^\gamma(\mathbb{T}^3))\times C([0,T],L^{\frac{2\gamma}{\gamma+1}}(\mathbb{T}^3))\times C([0,T];\R)$. Let $(\widetilde{\mf}_t^h)$ be the $\widetilde{\prst}$-augmented canonical filtration of the process $(\widetilde\varrho_h,\widetilde{\bf m}_h,\widetilde{W}, \widetilde{N}_h)$, that is 
\begin{equation*}
\begin{split}
\widetilde{\mf}_t^h&=\sigma\big(\sigma\big(\bfr_t\widetilde\varrho_h,\,\bfr_t\widetilde{\bf m}_h,\,\bfr_t \widetilde{W},  \bfr_t\widetilde{N}_h\big)\cup\big\{N\in\widetilde{\mf};\;\widetilde{\prst}(N)=0\big\}\big),\quad t\in[0,T],\\
\end{split}
\end{equation*}
where we denote by $\bfr_t$ the operator of restriction to the interval $[0,t]$ acting on various path spaces.
Let us remark that by assuming that the initial filtration $(\mathbb{F}_t)$ is the one generated by $W$, by \cite[Lemma A.6]{GA}, one can consider $(\widetilde{\mathbb{F}}_t^h)=(\widetilde{\mathbb{F}_t})$ is the filtration generated by $\widetilde{W}$.

\begin{Proposition}\label{prop:martsol}
For every $h\in(0,1)$, $\big((\widetilde{\Omega},\widetilde{\mf},(\widetilde{\mf}_{t}),\widetilde{\prst}),\widetilde\varrho_h,\widetilde{\bu}_h,\widetilde{W}\big)$ is a finite energy weak martingale solution to \eqref{eq:energy}--\eqref{eq:energy1} with the initial law $\Lambda_h$. 
\end{Proposition}

\begin{proof}
Proof of the above proposition directly follows form the Theorem 2.9.1 of the monograph by Breit et. al. \cite{BrFeHobook}.
\end{proof}

\noindent We note that the above proposition implies that the new random variables satisfy the following equations and the energy inequality on the new probability space:
\begin{itemize}
\item for all $\varphi\in C^\infty(\mathbb{T}^3)$ and $\bm{\varphi}\in C^\infty(\mathbb{T}^3)$ we have $\mathbb{P}$-a.s. for all $t\in[0,T]$
\begin{align}
\label{eq:energyt}
\langle \widetilde\varrho_{h}(t), \varphi\rangle &= \langle\widetilde\varrho_{h}(0) , \varphi\rangle - \int_0^{t}\langle \widetilde {\bf m}_{h}, \nabla \varphi\rangle\mathrm{d}s + \lambda\,h^3 \int_0^t  \langle \widetilde \varrho_{h}(t), \Delta \varphi\rangle \,ds+ \mathcal{\widetilde R}_1(h,\varphi) + \mathcal{\widetilde N}_1(h,\varphi)
\end{align}
\begin{align}
\langle \widetilde {\bf m}_{h}(t), \bm\varphi\rangle &= \langle \widetilde {\bf m}_{h}(0), \bm{\varphi}\rangle - \int_0^{t} \bigg\langle \bigg(\frac{\widetilde {\bf m}_{h}\otimes\widetilde {\bf m}_{h}}{\widetilde \varrho_{h}} + \widetilde p_h \mathbb{I} \bigg), \nabla  \bm{\varphi} \bigg \rangle\,\mathrm{d}s
+ \lambda\,h^N \int_0^t  \langle \widetilde {\mathbf{m}}_{h}(t), \Delta  \bm{\varphi}\rangle \,ds \notag \\
&+\int_0^{t}\langle\Psi(\widetilde \varrho_{h},\widetilde {\bf m}_{h}), \bm{\varphi}\rangle\, \mathrm{d}\widetilde W_{h} + \mathcal{\widetilde R}_2(h,\bm{\varphi}) + \mathcal{\widetilde N}_2(h,\bm{\varphi}), \label{eq:energyt1}
\end{align}
where $\mathcal{\widetilde R}_1(h,\varphi), \mathcal{\widetilde R}_2(h,\bm{\varphi}), \mathcal{\widetilde N}_1(h,\varphi)$, and $\mathcal{\widetilde N}_2(h,\bm{\varphi})$ are defined similarly as in \eqref{error}, in the new probability space.
\item the energy inequality\index{energy inequality} holds 
\begin{align}\label{EI2t''}
\begin{aligned}
&-\int_0^T \partial_t \psi \int_{\T^3} \bigg[ \frac{1}{2} \frac{ | \widetilde {\bf m}_{h} |^2 }{\widetilde\varrho_{h}} + \frac{\widetilde \varrho_{h}^\gamma}{\gamma-1} \bigg] \,dx \,ds 
\leq \psi(0) \int_{\T^3} \bigg[ \frac{1}{2}\frac{|\widetilde {\bf m}_{h}(0)|^2}{\widetilde\varrho_{h}(0)}  + \frac{ \widetilde \varrho_{h}^\gamma(0)}{\gamma-1}  \bigg] \dx \\
& \qquad \qquad  +\sum_{k=1}^\infty\int_0^T \psi \bigg(\int_{\T^3}\Psi_k (\widetilde \varrho_{h}, \widetilde {\bf m}_{h})\cdot \widetilde{\bf u}_{h}\dx\bigg){\rm d} \widetilde W_{h,k}
+ \frac{1}{2}\sum_{k = 1}^{\infty}  \int_0^T \psi
\int_{\T^3} \widetilde \varrho_{h}^{-1}| \Psi_k (\widetilde \varrho_{h}, \widetilde {\bf m}_{h}) |^2 \, dx\,{\rm d}s.
\end{aligned}
\end{align}
holds $\mathbb P$-a.s., for all $\psi\in C_c^\infty([0,T)),\,\psi\,\ge\,0.$
\end{itemize}

\noindent Next we would like to pass to the limit in $h$ in \eqref{eq:energyt}, \eqref{eq:energyt1}, and \eqref{EI2t''}. To do this, we first recall that a-priori estimates \eqref{aprhov}--\eqref{estrhou2} continue to hold for the new random variables. Thus, making use of the item $(5)$ of Proposition~\ref{prop:skorokhod1}, we conclude that $\widetilde \p$-a.s., 
\begin{align}\label{weak limit}
&\widetilde \varrho_h \rightharpoonup \langle {\mathcal{\widetilde V}^{\omega}_{t,x}}; \widetilde \varrho \rangle, \,\,\text{weakly in}\,\, L^{\gamma}((0,T)\times\T^3),\\
&\widetilde {\textbf m}_h \rightharpoonup \langle {\mathcal{\widetilde V}^{\omega}_{t,x}}; \widetilde {\textbf m} \rangle, \,\,\text{weakly in}\,\, L^{\frac{2\gamma}{\gamma+1}}((0,T)\times\T^3)\label{weak limit 1}.
\end{align}
In order to pass to the limit in the nonlinear terms present in the equations, we first introduce the corresponding concentration defect measures 
\begin{align*}
\widetilde \mu_{C}&= \widetilde C -\left\langle \mathcal{\widetilde V}^{\omega}_{(\cdot, \cdot)}; \frac{\widetilde {\bf m}\otimes \widetilde {\bf m}}{\widetilde\varrho} \right\rangle dxdt, \,\,
&\widetilde \mu_{P}= \widetilde P -\langle \mathcal{\widetilde V}^{\omega}_{(\cdot, \cdot)};p(\widetilde\varrho)\rangle dxdt, \\
\widetilde \mu_{E} &= \widetilde E- \left\langle \mathcal{\widetilde V}^{\omega}_{(\cdot, \cdot)}; \frac{1}{2} \frac{|\widetilde {\bf m}|^2}{\widetilde\varrho} +P(\widetilde\varrho) \right \rangle dx, \,\,
&\widetilde \mu_{D}= \widetilde D -\left\langle \mathcal{\widetilde V}^{\omega}_{(\cdot, \cdot)}; \sum_{k \geq 1} \frac{ |\Psi_k (\widetilde \varrho, \widetilde {\bf m}) |^2 }{\widetilde \varrho}\right\rangle dxdt.
\end{align*}
With the help of these concentration defect measures, thanks to the discussion in Subsection~\ref{ym}, we can conclude that $ \mathbb{\widetilde P}$-a.s.
\begin{align*}
&\widetilde C_h \rightharpoonup \left\langle \mathcal{\widetilde V}^{\omega}_{(\cdot, \cdot)}; \frac{\widetilde {\bf m}\otimes \widetilde {\bf m}}{\widetilde\varrho} \right\rangle dxdt + \widetilde \mu_{C}, \,\, \text{ weak-$*$ in}\, \, L^{\infty}(0,T; \mathcal{M}_b(\T^3)), \\
&\widetilde D_h \rightharpoonup \left\langle \mathcal{\widetilde V}^{\omega}_{(\cdot, \cdot)}; \sum_{k \geq 1} \frac{ |\Psi_k (\widetilde \varrho, \widetilde {\bf m}) |^2 }{\widetilde \varrho}\right\rangle dxdt + \widetilde \mu_{D}, \,\, \text{weak-$*$ in}\, \, L^{\infty}(0,T; \mathcal{M}_b(\T^3)), \\
&\widetilde E_h \rightharpoonup \left\langle \mathcal{\widetilde V}^{\omega}_{(\cdot, \cdot)}; \frac{1}{2} \frac{|\widetilde {\bf m}|^2}{\widetilde\varrho} +P(\widetilde\varrho) \right \rangle dx + \widetilde \mu_{E}, \,\, \text{weak-$*$ in}\, \, L^{\infty}(0,T; \mathcal{M}^+_b(\T^3)), \\
&\widetilde P_h \rightharpoonup \langle \mathcal{\widetilde V}^{\omega}_{(\cdot, \cdot)};P(\widetilde\varrho)\rangle dxdt + \widetilde \mu_{P}, \,\, \text{weak-$*$ in}\, \, L^{\infty}(0,T; \mathcal{M}_b(\T^3)). 
\end{align*}
Note that, collacting all the previous informations, we can pass to the limit in equation \eqref{eq:energyt} to get $\p$-a.s.
\begin{align*}
\int_{\T^3} \langle \widetilde{\mathcal{V}}^{\omega}_{\tau,x}; \widetilde{\varrho} \rangle \, \varphi\,\dx -\int_{\T^3} \langle \widetilde{\mathcal{V}}^{\omega}_{0,x}; \widetilde{\varrho} \rangle\, \varphi \,\dx 
= \int_{0}^{\tau} \int_{\T^3}  \langle \widetilde{\mathcal{V}}^{\omega}_{s,x}; \widetilde{\textbf{m}}\rangle \cdot \nabla_x \varphi \,dx \,\D s 
\end{align*}
holds for all $\tau\in [0,T)$, and for all $\varphi\in C^{\infty}(\T^3)$.

Next, we move onto the martingale term $\widetilde M_{h}:= \int_0^{t}\langle\Psi(\widetilde \varrho_{h},\widetilde {\bf m}_{h}), \bm{\varphi}\rangle\, \mathrm{d}\widetilde W_{h}$ coming from the momentum equation. Note that thanks to compact embedding given in Lemma~\ref{comp}, we conclude that for each $t$, $ \tilde M_{h}(t) \rightarrow \tilde M(t)$, $\p$-a.s. in the topology of $W^{-m,2}(\T^N)$. However, we are interested in identifying $\tilde M(t)$. Indeed, we may apply item $(5)$ of Proposition~\ref{prop:skorokhod1}, to the composition $\Psi_k (\widetilde \rho_{h},\widetilde {\bf m}_{h})$, $k \in \N$. This gives
$$
\Psi_k (\widetilde \rho_{h},\widetilde {\bf m}_{h}) \rightharpoonup
\Big \langle {\mathcal{\widetilde V}}^{\omega}_{t,x} ; \Psi_k (\widetilde \rho,\widetilde {\bf m}) \Big \rangle \,\, \mbox{weakly in} \,\,L^q((0,T)\times \T^3),
$$
$\p$-a.s., for some $q>1$. Moreover, for $m>3/2$, we have by Sobolev embedding 
\begin{align*}
\widetilde \E \Big[\int_0^T \| \Psi(\widetilde \rho_{h}, \widetilde {\bf m}_{h}) \|^2_{L_2(\mathcal{U}; W^{-m,2})}\,dt \Big]\le \widetilde \E \Big[\int_0^T (\widetilde \rho_{h})_{\T^3} \int_{\T^3} (\widetilde \rho_{h} + \widetilde \rho_{h} |\widetilde {\bf u}_{h}|^2)\,dx\,dt \Big]\le c(r).
\end{align*}
This implies that 
$$
\Psi_k (\widetilde \rho,\widetilde {\bf m}) \rightharpoonup  \Big \langle \mathcal{\widetilde V}^{\omega}_{t,x} ; \Psi_k (\widetilde \rho,\widetilde {\bf m}) \Big \rangle \,\, \mbox{weakly in} \,\,L^2(\Omega \times [0,T]; W^{-m,2}(\T^3)).
$$ 
Note that the It\^o integral
$$
I_t:\varphi \rightarrow \int_0^t \varphi(s) \D \widetilde {W}(s)
$$
is a linear and continuous (hence weakly continuous) map from $L^2(\Omega \times [0,T]; W^{-m,2}(\T^3))$ to $L^2(\Omega; W^{-m,2}(\T^3))$.
Therefore, we can make use of weak continuity of It\^{o} integral, and item $(4)$ of Proposition~\ref{prop:skorokhod1}, to conclude $I_t(\Psi_k (\widetilde \rho_h,\widetilde {\bf m}_h) )$ converges weakly to $I_t\big(\langle {\mathcal{\widetilde V}}^{\omega}_{t,x} ; \Psi_k (\widetilde \rho,\widetilde {\bf m}) \rangle\big)$ in $L^2(\Omega;W^{-m,2}(\mathbb{T}^3))$. Collectting all above informations, we can conclude that 
\begin{equation} \label{second condition measure-valued solution 23}
\begin{aligned}
&\int_{\Omega}\Big(\int_{\T^3} \langle \widetilde{\mathcal{V}}^{\omega}_{\tau,x};\widetilde{\textbf{m}}\rangle \cdot \bm{\varphi}\, dx - \int_{\T^3} \langle \widetilde{\mathcal{V}}^{\omega}_{0,x};\widetilde{\textbf{m}}\rangle \cdot \bm{\varphi}\, \D x\Big)\alpha(\omega)\, \D \tilde{\p}(\omega) \\
&\qquad = \int_{\Omega}\Big(\int_{0}^{\tau} \int_{\T^3} \left[\left\langle \widetilde{\mathcal{V}}^{\omega}_{t,x}; \frac{\widetilde{\textbf{m}}\otimes \widetilde{\textbf{m}}}{\widetilde\varrho} \right\rangle: \nabla{\bm{\varphi}} + \langle \widetilde{\mathcal{V}}^{\omega}_{t,x};p(\widetilde{\varrho})\rangle \divv \bm{\varphi} \right] dx dt \\
&\qquad \qquad+ \int_{\T^3} \bm{\varphi}\,\int_0^{\tau} \left\langle \widetilde{\mathcal{V}}^{\omega}_{t,x}; \Psi(\widetilde{\varrho},\widetilde{\textbf{m}} )\right\rangle \, \D \widetilde{W} \,dx+ \int_{0}^{\tau} \int_{\T^3}  \nabla \bm{\varphi}: d(\widetilde{\mu}_C+\widetilde{\mu}_{P}\mathbb{I})\Big)\alpha(\omega)\,\D\widetilde{\p}(\omega),
\end{aligned}
\end{equation}
holds for all $\tau\in [0,T)$, for all $\alpha\in L^2(\widetilde{\Omega})$ and for all $\bm{\varphi} \in C^{\infty}(\T^3;\mathbb{R}^3)$. Since $C^\infty(\mathbb{T}^3)$ is separable space with sup norm, above equality \eqref{second condition measure-valued solution 23} implies that $\p$-a.s.
\begin{equation*} 
\begin{aligned}
\int_{\T^3} \langle \widetilde{\mathcal{V}}^{\omega}_{\tau,x};\widetilde{\textbf{m}}\rangle \cdot \bm{\varphi}\, dx - \int_{\T^3} \langle \widetilde{\mathcal{V}}^{\omega}_{0,x};\widetilde{\textbf{m}}\rangle \cdot \bm{\varphi}\, \D x 
&=\int_{0}^{\tau} \int_{\T^3} \left[\left\langle \widetilde{\mathcal{V}}^{\omega}_{t,x}; \frac{\widetilde{\textbf{m}}\otimes \widetilde{\textbf{m}}}{\widetilde{\varrho}} \right\rangle: \nabla{\bm{\varphi}} + \langle \widetilde{\mathcal{V}}^{\omega}_{t,x};p(\widetilde{\varrho})\rangle \divv \bm{\varphi} \right] dx dt \\
&\qquad \qquad+ \int_{\T^3} \bm{\varphi}\,\int_0^{\tau} \left\langle \widetilde{\mathcal{V}}^{\omega}_{t,x}; \Psi(\widetilde{\varrho},\widetilde{\textbf{m}} )\right\rangle \, \D \widetilde{W} \,dx+ \int_{0}^{\tau} \int_{\T^3}  \nabla \bm{\varphi}: d(\widetilde{\mu}_C+\widetilde{\mu}_{P}\mathbb{I})
\end{aligned}
\end{equation*}
holds for all $\tau\in [0,T)$, and for all $\bm{\varphi} \in C^{\infty}(\T^3;\mathbb{R}^3)$. where $(\widetilde{\mu}_C+\widetilde{\mu}_{P}\mathbb{I})\in L^{\infty}_{w^*}\big([0,T]; \mathcal{M}_b({\T^3} ), \tilde{\p}$-a.s., is tensor-valued measure. Therefore we conclude that \eqref{first condition measure-valued solution}-\eqref{second condition measure-valued solution} holds.

\noindent  Regarding the convergence of martingale term $\widetilde N_{h}$, appearing in the energy inequality, we have following proposition.
\begin{Proposition}
	For each $t$, $\widetilde N_{h}(t) \rightarrow \widetilde N(t)$ in $\R$, $\p$-a.s., and $\widetilde N(t)$ is a real valued square-integrable martingale.
\end{Proposition}

\begin{proof}
	Note that, thanks to Proposition~\ref{prop:skorokhod1}, we have the information $\widetilde N_{h} \rightarrow  \widetilde N $ $\p$-a.s. in $C([0,T]; \R)$. To conclude that $\widetilde N(t)$ is a martingale, We have to show that, $\p$-a.s.
	$$
	\widetilde \E[\widetilde N(t)| \mathcal{\widetilde F}_s] = \widetilde N(s),
	$$
	for all $t,s \in [0,T]$ with $s \le t$. To prove this, it is sufficient to show that, for all $A\in\mathcal{F}_s$
	$$
	\widetilde \E \Big[ \mathcal{I}_{A} \big(\widetilde N(t)-\widetilde N(s)\big) \Big]=0,
	$$
 Now using the fact that $\widetilde N_{h}(t)$ is a martingale, we know that
	$$
	\widetilde \E \Big[ \mathcal{I}_A \big(\widetilde N_{h}(t)-\widetilde N_{h}(s)\big) \Big]=0,
	$$
	 for all $A\in\mathcal{F}_s$. Note that for all $t\in[0,T]$ $\widetilde N_{h}(t)$ is uniform bounded in $ L^2(\widetilde \Omega)$, with the help of Vitali's convergence theorem, we can pass to the limit in $h$ to conclude that $\widetilde N(t)$ is a martingale. 
\end{proof}

\begin{Lemma}\label{rhoutight131}
	The concentration defect $0\le \mathcal{\widetilde D}(\tau):= \widetilde \mu_{E}(\tau)(\T^3)$ dominates defect measures $\widetilde \mu_{D}$ in the sense of Lemma~\ref{lemma001}. More precisely, there exists a constant $C>0$ such that
	\begin{equation*} 
	\int_{0}^{\tau} \int_{\T^3} d|\widetilde \mu_C| + \int_{0}^{\tau} \int_{\T^3} d|\widetilde \mu_D| + \int_{0}^{\tau} \int_{\T^3} d|\widetilde \mu_P| \leq C \int_{0}^{\tau} \mathcal{\widetilde D}(\tau)\,dt,
	\end{equation*}	
	for a.e. $\tau \in (0,T)$, $\p$-a.s.
\end{Lemma}

\begin{proof}
	Following deterministic argument we can conclude that $\widetilde \mu_{E}$ dominates defect measures $\widetilde \mu_{C}, \widetilde \mu_P$. To show the dominance of $\widetilde \mu_{E}$ over $\widetilde \mu_{D}$, observe that by virtue of hypotheses (\ref{FG1}), (\ref{FG2}), the function
	\[
	[\vr, \vc{m}] \mapsto \sum_{k \geq 1} \frac{ | \Psi_k (\varrho, {\bf m}) |^2 }{\varrho} \ \mbox{is continuous},
	\]
	and as such dominated by the total energy
	\[
	\sum_{k \geq 1} \frac{ | \Psi_k (\varrho, {\bf m}) |^2 }{\varrho} \leq c \left(  \varrho  + \frac{| {\bf m} |^2}{\vr}  \right) \leq c \left( \frac{1}{2} \frac{|{\bf m}|^2}{\vr}  + P(\vr) \right) + 1.
	\]
	Hence, a simple application of the Lemma~\ref{lemma001} finishes the proof of the lemma. 
\end{proof}

To conclude  \eqref{third condition measure-valued solution}, we proceed as follows. First note that we can pass to limit in $h\to 0$ in \eqref{EI2t''} to obtain the following energy inequality in the new probablity space.

\begin{align}\label{tr}
	\begin{aligned}
		-\int_0^T \partial_t \psi \bigg[\int_{\mathbb T^3} & \bigg\langle\widetilde{\mathcal{V}}_{t,x}^\omega;\frac{|\widetilde{\bf m}|^2}{\widetilde\varrho} + P(\widetilde\varrho)\bigg\rangle dx+ \widetilde{\mathcal{D}}(s)\bigg]\,ds
		\leq \psi(0) \int_{\mathbb T^3} \bigg[\bigg\langle\widetilde{\mathcal{V}}_{t,x}^\omega; \frac{1}{2}\frac{|\widetilde{\bf m}|^2}{\widetilde\varrho}  + P(\widetilde\varrho(0))\bigg\rangle\bigg] \dx \\
		&\qquad +\frac{1}{2}\int_0^T\psi\int_{\mathbb{T}^3}\D \tilde{\mu}_{{D}}+\int_0^T\psi\, \D\widetilde{N} + \frac{1}{2}\sum_{k = 1}^{\infty}  \int_0^T \psi
		\int_{\mathbb T^3} \bigg\langle\widetilde{\mathcal{V}}_{t,x};{\widetilde\varrho}^{-1}|  \Psi_k (\widetilde\varrho, \widetilde{\bf m}) |^2 \bigg\rangle\,dx\, {\rm d}s.
	\end{aligned}
\end{align}
holds $\mathbb P$-a.s., for all $\psi\in C_c^\infty([0,T)),\,\psi\,\ge\,0.$
Fix any $s$ and $t$ such that
$0\,\textless\,s\,\textless\,t\,\textless\,T$. For any $r\,\textgreater\,0$ with $0\,\textless\,s-r\textless\,t+r\,\textless\,T$, let $\psi_r$ be a Lipschitz fucntion that is linear on $[s-r,s]$ or $[t, t+r]$ and satiesfies
$$\psi_r(\tau)=\begin{cases}
	0,&\text{if}\,\, \tau\in[0,s-r]\,\,\text{or}\,\,\tau\in[t+r,T]\\
	1,&\text{if}\,\,\tau\in[s,t].
\end{cases}$$
Then, $\psi_r$ is an admissible test fuction in \eqref{tr}, via a standard regularization argument. From \eqref{tr} with $\psi_r$ as test fuction, we have $\p$-a.s, for all $t\in[0,T]$
\begin{align}\label{pr}
	&\frac{1}{r}\int_{t}^{t+r}\bigg(\int_{\mathbb{T}^3}\left\langle \mathcal{\tilde{V}}^{\omega}_{\tau,x};\frac{1}{2}\frac{|{\widetilde{\bf m}}|^2}{{\widetilde\varrho}}+P(\widetilde\varrho) \right\rangle \dx +\mathcal{\widetilde{D}}(s)\bigg){\rm d}\tau\,\notag\\&\qquad\le\,\frac{1}{r}\int_{t-r}^t\bigg(\int_{\mathbb{T}^3}\left\langle \mathcal{\widetilde{V}}^{\omega}_{\tau,x};\frac{1}{2}\frac{|\widetilde{{\bf m}}|^2}{{\widetilde{\varrho}}}+P(\widetilde{\varrho}) \right\rangle \dx+\mathcal{\widetilde{D}}(s)\bigg){\rm d}\tau+\frac{1}{2} \int_{s-r}^{t+r}\psi_r(\tau) \int_{\T^3} \left\langle \widetilde{\mathcal{V}}^{\omega}_{\tau,x};{\widetilde{\varrho}^{-1}|{\Psi_k(\widetilde{\varrho},\widetilde{\mathbf{m}})}|^2} \right\rangle \dx\notag\\&\qquad\qquad +\frac{1}{2}\int_{s-r}^{t+r}\int_{\mathbb{T}^3}\psi_r(\tau){\rm d}\widetilde\mu_D(x,\tau) +\int_{s-r}^{t+r}\psi_r(\tau){\rm d}\widetilde N
\end{align}
Now letting limit as $r\to 0^+$ in \eqref{pr}, then we have  $\p$-a.s,for all $t\in[0,T]$
\begin{align*}
	&\lim_{r\to 0^+}\frac{1}{r}\int_{t}^{t+r}\bigg(\int_{\mathbb{T}^3}\left\langle \mathcal{\widetilde{V}}^{\omega}_{\tau,x};\frac{1}{2}\frac{|\widetilde{{\bf m}}|^2}{\varrho}+P(\widetilde{\varrho}) \right\rangle \dx +\mathcal{\widetilde{D}}(s)\bigg){\rm d}\tau\,\notag\\&\qquad\le\,\lim_{r\to0 ^+}\frac{1}{r}\int_{t-r}^t\bigg(\int_{\mathbb{T}^3}\left\langle \mathcal{\widetilde{V}}^{\omega}_{\tau,x};\frac{1}{2}\frac{|{{\widetilde{\bf m}}|^2}}{\widetilde{\varrho}}+P(\widetilde{\varrho}) \right\rangle \dx+\mathcal{\widetilde{D}}(s)\bigg){\rm d}\tau+\frac{1}{2} \int_{s}^{t}\int_{\T^3} \left\langle \mathcal{V}^{\omega}_{\tau,x};{\widetilde{\varrho}^{-1}|{\Psi_k(\widetilde{\varrho},\widetilde{\mathbf{m}})}|^2} \right\rangle \dx\,{\rm d}\tau\notag\\&\qquad\qquad +\frac{1}{2}\int_{s}^{t}\int_{\mathbb{T}^3}{\rm d}\widetilde\mu_D(x,\tau) +\int_{s}^{t}{\rm d}\widetilde N
\end{align*}
Thus we conclude that \eqref{third condition measure-valued solution} holds. If $s=0$, then we need different test function to conclude result. In this case we take 
$$\psi_r(\tau)=\begin{cases}
	1,&\text{if}\,\,\tau\in[0,t]\\
	\text{linear},&\text{if}\,\,\tau\in[t,t+r]\\
	0,&\text{Otherwise}.
\end{cases}$$
and apply the same argument as before.


\section{Weak-Strong Uniqueness Principle}
\label{w-s}
In this section, we establish pathwise weak (measure-valued)--strong uniqueness principle for dissipative measure-valued martingale solutions. In what follows, we first introduce the relative energy functional which plays a pivotal role in the proof of weak (measure-valued)--strong uniqueness principle. In the context of compressible Euler equations, relative energy functional reads
\begin{align}\label{relative energy 1}
\mathfrak{E}'_{\text{mv}}\big(\varrho,\mathbf{m}\,\big|\,s,\mathbf{Q}\big)(t):&=\int_{\mathbb{T}^3}\bigg\langle {\mathcal{V}^{\omega}_{t,x}}; \frac{1}{2} { \frac{|\textbf{m}|^2}{\varrho}} + P(\varrho)\bigg\rangle \D x -\int_{\mathbb{T}^3}\big\langle {\mathcal{V}^{\omega}_{t,x}}; {\textbf{m}}\big\rangle\cdot\mathbf{Q} \D x+\frac{1}{2}\int_{\mathbb{T}^3}\big\langle {\mathcal{V}^{\omega}_{t,x}};\varrho \big\rangle|\mathbf{Q}|^2\D x\notag\\&\qquad -\int_{\mathbb{T}^3}\big\langle {\mathcal{V}^{\omega}_{t,x}};\varrho \big\rangle P'(s)\,\D x-\int_{\mathbb{T}^3}[P'(s)s-P(s)]\D x + \mathcal{D}(t).
\end{align}
In view of the energy inequality \eqref{energy_001}, it is clear that the above energy functional \eqref{relative energy 1} is defined for all $t\in[0,T]\setminus\mathcal{A}$, where the set $\mathcal{A}$, may depends on $\omega$, has Lebesgue measure zero. We also define relative energy function for all $t\in\mathcal{A}$ as follows
\begin{align}\label{relative energy 2}
\mathfrak{E}''_{\text{mv}}\big(\varrho,\mathbf{m}\,\big|\,s,\mathbf{Q}\big)(t):&=\lim_{r\to 0}\int_{t}^{t+r}\bigg[\int_{\mathbb{T}^3}\bigg\langle {\mathcal{V}^{\omega}_{s,x}}; \frac{1}{2} { \frac{|\textbf{m}|^2}{\varrho}} + P(\varrho)\bigg\rangle \D x +\mathcal{D}(s)\bigg]\D s -\int_{\mathbb{T}^3}\big\langle {\mathcal{V}^{\omega}_{t,x}}; {\textbf{m}}\big\rangle\cdot\mathbf{Q} \D x\notag\\&\qquad +\frac{1}{2}\int_{\mathbb{T}^3}\big\langle {\mathcal{V}^{\omega}_{t,x}};\varrho \big\rangle|\mathbf{Q}|^2\D x-\int_{\mathbb{T}^3}\big\langle {\mathcal{V}^{\omega}_{t,x}};\varrho \big\rangle P'(s)\,\D x-\int_{\mathbb{T}^3}[P'(s)s-P(s)]\D x + \mathcal{D}(t).
\end{align}
Using relative energy functionals \eqref{relative energy 1}-\eqref{relative energy 2}, we define relative energy functional for all time $t\in[0,T]$ as follows
\begin{align}
\mathfrak{E}_{\mathrm{mv}}\big(\varrho,\mathbf{m}\,\big|\,s,\mathbf{Q}\big)(t):=\begin{cases}
\mathfrak{E}'_{\mathrm{mv}}\big(\varrho,\mathbf{m}\,\big|\,s,\mathbf{Q}\big)(t),&\text{if}\,\,t\in[0,T]\setminus\mathcal{A}\\
\mathfrak{E}''_{\mathrm{mv}}\big(\varrho,\mathbf{m}\,\big|\,s,\mathbf{Q}\big)(t),&\text{if}\,\,t\in\mathcal{A}
\end{cases}
\end{align}
With the help of the above definition of relative energy functional, we are now in a position to derive the following relative energy inequality.
\begin{Proposition}[Relative Energy Inequality] \label{prop:01}
	Let $\big[ \big(\Omega,\mf, (\mf_{t})_{t\geq0},\mathbb{P} \big); \mathcal{V}^{\omega}_{t,x}, W \big]$ be a dissipative measure-valued martingale solution to the system \eqref{P1}--\eqref{P2}. Suppose $(s,\mathbf{Q})$ be a pair of stochastic processes which are adapted to the filtration $(\mf_{t})_{t\ge\,0}$ and which satisfies
	$$\D s=s_1\D t+s_2 \D W,$$
	$$\D \mathbf{Q}=\mathbf{Q}_1 \D t+ \mathbf{Q}_2 \D W$$
	with $$s\in C([0,T],W^{1,q}(\mathbb{T}^3)),\,\,\mathbf{Q}\in C([0,T];W^{1,q}(\mathbb{T}^3))\,\,\,\,\p-\text{a.s.}$$
	$$\mathbb{E}\bigg[\sup_{t\in[0,T]}\|s\|_{W^{1,q}(\mathbb{T}^3)}^2\bigg]^q + \mathbb{E}\bigg[\sup_{t\in[0,T]}\|\mathbf{Q}\|_{W^{1,q}}^2\bigg]^q\,\le\,C\,\, \text{for all}\,\,\, 2\le\,q\,\textless\,\infty,$$
	$$0\textless\,r_1\le\,s(t,x)\le\,r_2\,\,\,\p-\text{a.s.}$$
	Moreover, $s_i,\mathbf{Q}_i$ for $i=1,2$, satisfy
	$$s_1,\mathbf{Q}_1\in L^q(\Omega;L^q(0,T; W^{1,q}(\mathbb{T}^3)))\,\,\,\,s_2, \mathbf{Q}_2\in L^2(\Omega;L^2((0,T);L_2(\mathfrak{U};L^2(\mathbb{T}^3))),$$
	$$\bigg(\sum_{k\ge\,1}|s_2(e_k)|^q\bigg)^{1/q},\,\bigg(\sum_{k\ge\,1}|\mathbf{Q}_2(e_k)|^q\bigg)^{1/q}\in L^q(\Omega;L^q(0,T;L^q(\mathbb{T}^3)))$$
	Then the following relative energy inequality holds $\p$-a.s., for all $t\in[0,T]$
	\begin{align}\label{relative energy}
	\mathfrak{E}_{\mathrm{mv}}\big(\varrho,\mathbf{m}\,\big|\,s,\mathbf{Q}\big)(t)\le\,\mathfrak{E}_{\mathrm{mv}}\big(\varrho,\mathbf{m}\,\big|\,s,\mathbf{Q}\big)(0)+\mathcal{M}_{RE}(t)+\int_{0}^t\mathfrak{R}_{\mathrm{mv}}\big(\varrho,\mathbf{m}\,\big|\,s,\mathbf{Q}\big)(\tau)\,\D \tau
	\end{align}
	where
	\begin{align*}
	\begin{aligned}
	\mathfrak{R}_{\mathrm{mv}}\big(\varrho,\mathbf{m}\,\big|\,s,\mathbf{Q}\big)(t)&=\int_{\mathbb{T}^3}\big\langle \mathcal{V}_{t,x}^\omega;\varrho\mathbf{Q}-\mathbf{m}\big\rangle\cdot[\mathbf{Q}_1+\nabla\mathbf{Q}\cdot\mathbf{Q}]\D x +\int_{\mathbb{T}^3}\big\langle \mathcal{V}_{t,x}^\omega;\frac{(\mathbf{m}-\varrho\mathbf{Q})\otimes(\varrho\mathbf{Q}-\mathbf{m})}{\varrho}\big\rangle:\nabla \mathbf Q \D x\\&\qquad + \int_{\mathbb{T}^3}[(s-\big\langle \mathcal{V}_{t,x}^\omega;\varrho\big\rangle P''(s) s_1 + \nabla_x P'(s)\cdot(s\mathbf{Q}-\big\langle\mathcal{V}_{t,x}^\omega);\mathbf{m}\big\rangle)]\D x\\&\qquad+\int_{\mathbb{T}^3}[p(s)-\big\langle\mathcal{V}_{t,x}^\omega; p(s)\big\rangle]\Div(\mathbf{Q})\D x + \frac{1}{2}\sum_{k\ge\,1}\int_{\mathbb{T}^3}\bigg\langle\mathcal{V}_{t,x}^\omega;\varrho\bigg|\frac{\Psi_k(\varrho,\mathbf{m})}{\varrho}-\mathbf{Q}_2(e_k)\bigg|^2\bigg\rangle\D x\\
	&\qquad+\frac{1}{2}\sum_{k\ge\,1}\int_{\mathbb{T}^3}\big\langle\mathcal{V}_{t,x}^\omega;\varrho\big\rangle P''(s)|s_2(e_k)|^2\,\D x + \frac{1}{2}\sum_{k\ge\,1}\int_{\mathbb{T}^3}p''(s)|s_2(e_k)|^2\,\D x \\&\qquad
	-\int_{\mathbb{T}^3}\nabla\mathbf{Q}:\D \mu_m +\frac{1}{2}\int_{\mathbb{T}^3}d\mu_e. 
	\end{aligned}
	\end{align*}
Here $\mathcal{M}_{RE}$ is a real valued square integrable matingale.
	\noindent
	\begin{proof}
The proof of this proposition is a consequence of generalized It\^o formula, which is similar to the Lemma 4.1 in \cite{MKS01}. However, strictly speaking, the proof given in \cite{MKS01} is based on a slightly different notion of dissipative measure-valued martingale solutions. Therefore, for the sake of completness, we briefly mention the proof. Note that given condtions on stochastic process in this propostion allows us to apply It\^o formula to compute $\int_{\mathbb{T}^3}\big\langle\mathcal{V}_{t,x}^\omega;\mathbf{m}\big\rangle\cdot\mathbf{Q}\D x$. The result is
\begin{align}\label{q1}
&\D \bigg(\int_{\mathbb{T}^3}\big\langle\mathcal{V}_{t,x}^\omega;\mathbf{m}\big\rangle\cdot\mathbf{Q}\D x\bigg)=\int_{\mathbb{T}^3}\bigg(\big\langle\mathcal{V}_{t,x}^\omega;\mathbf{m}\big\rangle\cdot\mathbf{Q}_1 + \bigg\langle\mathcal{V}_{t,x}^\omega;\frac{\mathbf{m}\otimes\mathbf{m}}{\varrho}\bigg\rangle:\nabla\mathbf{Q}+\big\langle\mathcal{V}_{t,x}^\omega;p(\varrho)\big\rangle\Div\mathbf{Q}\bigg)\D x\D t \notag\\&\qquad+ \sum_{k\,\ge\,1}\int_{\mathbb{T}^3}\mathbf{Q}_2(e_k)\cdot\big\langle\mathcal{V}_{t,x}^\omega;\Psi_k(\varrho,\mathbf{m})\big\rangle\D x\D t +\int_{\mathbb{T}^3}\nabla\mathbf{Q}:\D\mu_m\D t + \int_{\mathbb{T}^3}\mathbf{Q}\cdot\big\langle\mathcal{V}_{t,x}^\omega;\Psi(\varrho,\mathbf{m})\big\rangle\D W\notag\\&\qquad+\int_{\mathbb{T}^3}\big\langle \mathcal{V}_{t,x}^\omega;\mathbf{m}\big\rangle\cdot\mathbf{Q}_2\D x \D W.
\end{align}
Similary, we get
\begin{align}\label{q2}
\D \bigg(\int_{\mathbb{T}^3}\frac{1}{2}\big\langle\mathcal{V}_{t,x}^\omega;\varrho\big\rangle|\mathbf{Q}|^2\D x\bigg)&= \int_{\mathbb{T}^3}\big\langle\mathcal{V}_{t,x}^\omega;\mathbf{m}\big\rangle\cdot\nabla\mathbf{Q}\cdot\mathbf{Q}\D x\D t+\int_{\mathbb{T}^3}\big\langle\mathcal{V}_{t,x}^\omega;\varrho\big\rangle\mathbf{Q}\cdot\mathbf{Q_1}\D x\D t \\\notag&\qquad+\frac{1}{2}\sum_{k\ge\,1} \int_{\mathbb{T}^3}\big\langle\mathcal{V}_{t,x}^\omega;\varrho\big\rangle|\mathbf{Q}_2(e_k)|^2\D x\D t +\int_{\mathbb{T}^3}\big\langle\mathcal{V}_{t,x}^\omega;\varrho\big\rangle\mathbf{Q}\cdot\mathbf{Q}_2\D x\D W,
\end{align}
and
\begin{align}\label{q3}
\D \bigg(\int_{\mathbb{T}^3}\big(P'(s)s-P(s)\big)\D x\bigg)=\int_{\mathbb{T}^3}p'(s) s_1 \D x\D t +\frac{1}{2}\sum_{k\,\ge\,1}\int_{\mathbb{T}^3}p''(s)|s_2(e_k)|^2\,\D x\D t +\int_{\mathbb{T}^3}p'(s) s_2\D x\D W,
\end{align}
and
\begin{align}\label{q4}
\D \bigg(\int_{\mathbb{T}^3}\big\langle\mathcal{V}_{t,x}^\omega;\varrho\big\rangle P'(s)\D x\bigg)&= \int_{\mathbb{T}^3}\big\langle\mathcal{V}_{t,x}^\omega;\mathbf{m}\big\rangle\cdot\nabla_x P'(s)\D x\D t +\int_{\mathbb{T}^3}\big\langle\mathcal{V}_{t,x}^\omega;\varrho\big\rangle P''(s) s_1 \D x \D t \notag\\&\qquad+ \frac{1}{2}\sum_{k\ge\,1}\int_{\mathbb{T}^3}\big\langle\mathcal{V}_{t,x}^\omega;\varrho\big\rangle P''(s)|s_2(e_k)|^2\D x\D t + \int_{\mathbb{T}^3}\big\langle\mathcal{V}_{t,x}^\omega;\varrho\big\rangle P''(s)s_2 \D x\D W.
\end{align}
Now we can combine \eqref{q1}-\eqref{q4} with \eqref{third condition measure-valued solution} and define $\mathcal{M}_{RE}$ be the sum of all martingale terms which come from \eqref{q1}-\eqref{q4} and  \eqref{third condition measure-valued solution} to obtain \eqref{relative energy}. 
\end{proof}
With the help of the Proposition~\ref{prop:01}, we now briefly describe the proof of the weak (measure-valued)--strong uniqueness principle. For detials of the proof, we refer to \cite[Chapter 6]{BrFeHobook} and \cite{MKS01}.
\end{Proposition}
\begin{Theorem}[Weak-Strong Uniqueness] \label{Weak-Strong Uniqueness_01}
	Let $\big[ \big(\Omega,\mf, (\mf_{t})_{t\geq0},\mathbb{P} \big); \mathcal{V}^{\omega}_{t,x}, W \big]$ be a dissipative measure-valued martingale solution to the system \eqref{P1}--\eqref{P2}. On the same stochastic basis $\big(\Omega,\mf, (\mf_{t})_{t\geq0},\mathbb{P} \big)$, let us consider the unique maximal strong pathwise solution to the Euler system (\ref{P1}--\ref{P2}) given by 
	$(\bar{\varrho},\bar{{\bf u}},(\mathfrak{t}_R)_{R\in\mn},\mathfrak{t})$ driven by the same cylindrical Wiener process $W$ with the initial data $\bar{\varrho}(0), \bar{\varrho}\bar{\bf u}(0)$ satisfies 
	\begin{equation*} 
	\mathcal{V}^{\omega}_{0,x}= \delta_{\bar{\varrho}(0,x),(\bar{\varrho}\bar{\bf u})(0,x)}, \,\p-\mbox{a.s.,}\, \mbox{for a.e. } x \in \T^3.
	\end{equation*}
	Then a.e. $t\in[0,T]$ $\mathcal{D}(t\wedge\mathfrak{t}_R)=0$, $\p$-a.s., and $\p-\mbox{a.s.,}$
	\begin{equation}\label{uniqueness_01}
	\mathcal{V}^{\omega}_{t \wedge \mathfrak{t}_R,x}	= \delta_{\bar{\varrho}(t \wedge \mathfrak{t}_R,x), (\bar{\varrho}\bar{\bf u})(t \wedge \mathfrak{t}_R,x)}, \,\,  \mbox{for a.e. }(t,x)\in (0,T)\times \T^3.
	\end{equation}
\end{Theorem} 
\begin{proof} 
Since $(\bar{\varrho}(\cdot\wedge\mathfrak{t}_R),\bar{\mathbf{u}})$ is the strong pathwise solution to stystem \eqref{P1}-\eqref{P2}, so we can replace  $(s,\mathbf{Q})$ by $(\bar{\varrho}(\cdot\wedge\mathfrak{t}_R),\bar{\mathbf{u}})$ in the relative energy inequality \eqref{relative energy}. Then we have $\p$-a.s., for all $t\in[0,T]$,
\begin{align}\label{relative energy 3}
\mathfrak{E}_{\mathrm{mv}}\big(\varrho,\mathbf{m}\,\big|\,\bar{\varrho},\bar{\mathbf{u}}\big)(t\wedge\mathfrak{t}_R)\le\,\mathfrak{E}_{\mathrm{mv}}\big(\varrho,\mathbf{m}\,\big|\,\bar{\varrho},\bar{\mathbf{u}}\big)(0)+\mathcal{M}_{RE}(t\wedge\mathfrak{t}_R)+\int_{0}^{t\wedge\mathfrak{t}_R}\mathfrak{R}_{\mathrm{mv}}\big(\varrho,\mathbf{m}\,\big|\,\bar{\varrho},\bar{\mathbf{u}}\big)(s)\D s,
\end{align}
where $\mathfrak{R}_{\mathrm{mv}}\big(\varrho,\mathbf{m}\,\big|\,\bar{\varrho},\bar{\mathbf{u}}\big)$ is given by \eqref{reminder} after replacing $(s,\mathbf{Q})$ by $(\bar{\varrho}(\cdot\wedge\mathfrak{t}_R),\bar{\mathbf{u}})$. Following \cite{MKS01} and \cite[Chapter 6]{BrFeHobook}, one can verify that
\begin{align}\label{reminder 2}
\int_{0}^{t\wedge\mathfrak{t}_R}\mathfrak{R}_{\mathrm{mv}}\big(\varrho,\mathbf{m}\,\big|\,\bar{\varrho},\bar{\mathbf{u}}\big)(s)\,\D s\le\,c(R)\int_{0}^{t\wedge\mathfrak{t}_R}\big(\mathfrak{E}_{\mathrm{mv}}\big(\varrho,\mathbf{m}\,\big|\,\bar{\varrho},\bar{\mathbf{u}}\big)(s)\, \D s.
\end{align}
In light of \eqref{relative energy 3} and \eqref{reminder 2}, a straightforward consequence of Gronwall's lemma yields, for all $t\in[0,T]$
\begin{align*}
\mathbb{E}\big[\mathfrak{E}_{\mathrm{mv}}\big(\varrho,\mathbf{m}\,\big|\,\bar{\varrho},\bar{\mathbf{u}}\big)(t\wedge\mathfrak{t}_R)\big]\le c(R)\,\mathbb{E}\big[\mathfrak{E}_{\mathrm{mv}}\big(\varrho,\mathbf{m}\,\big|\,\bar{\varrho},\bar{\mathbf{u}}\big)(0)\big].
\end{align*}
Since initial data are same for both solutions, right hand side of above inequality equals to zero. Therefore it implies that for all $t\in [0,T]$
$$\mathbb{E}\big[\mathfrak{E}_{\mathrm{mv}}\big(\varrho,\mathbf{m}\,\big|\,\bar{\varrho},\bar{\mathbf{u}}\big)(t\wedge\mathfrak{t}_R)\big]=0.$$
This also implies that
$$\lim_{r\to 0}\int_{t}^{t+r}\mathbb{E}\big[\mathfrak{E}_{\mathrm{mv}}\big(\varrho,\mathbf{m}\,\big|\,\bar{\varrho},\bar{\mathbf{u}}\big)(s\wedge\mathfrak{t}_R)\big]\D s=0.$$
In view of a priori estimates, a usual Lebesgue point argument, and application of Fubini's theorem reveals that for a.e. $t\in[0,T]$,
$$\mathbb{E}\big[\mathfrak{E}'_{\mathrm{mv}}\big(\varrho,\mathbf{m}\,\big|\,\bar{\varrho},\bar{\mathbf{u}}\big)(t\wedge\mathfrak{t}_R)\big]=0$$
Since the defect measure $\mathcal{D}\ge\,0$, we have for a.e. $t\in[0,T]$, $\mathcal{D}(t\wedge\mathfrak{t}_R)=0$, $\p$-a.s. Moreover, $\p-\mbox{a.s.}$
\begin{equation*}
\mathcal{V}^{\omega}_{t \wedge \mathfrak{t}_R,x}	= \delta_{\bar{\varrho}(t \wedge \mathfrak{t}_R,x), (\bar{\varrho}\bar{\bf u})(t \wedge \mathfrak{t}_R,x)}, \,\,  \mbox{for a.e. }(t,x)\in (0,T)\times \T^3.
\end{equation*} 
This finishes the proof of the theorem.
\end{proof}

\section{Proof of Theorem~\ref{dissipative solution}: Convergence to dissipative solution}
\label{dissipative solution 1}

In view of the Proposition~\ref{prop:skorokhod1} and convergence results given by \eqref{weak limit}-\eqref{weak limit 1}, we conclude that there is subsequence $\{(\vr_{h_k}(t), \vm_{h_k}(t))\}_{h_k>0}$ such that $\p$-a.s,
	$$\varrho_{h_k}\to\langle {\mathcal{V}^{\omega}_{t,x}}; \varrho \rangle\,\,\mbox{in}\,\,\,C_w([0,T],L^\gamma(\mathbb{T}^3)),$$
$$\mathbf{m}_{h_k}\to\langle {\mathcal{V}^{\omega}_{t,x}}; {\textbf m} \rangle\,\,\mbox{in}\,\,\,C_w([0,T],L^{\frac{2\gamma}{\gamma+1}}(\mathbb{T}^3)).$$ 
For the pointwise converegnce of numerical approximations, we can make use of Proposition~\ref{kconvergence}. Indeed, we obtain $\p$-a.s., there exists a subsequece $\{(\vr_{h_k}(t), \vm_{h_k}(t))\}_{h_k>0}$ such that
\begin{equation*}
	\begin{aligned}
		\frac 1N \sum_{k=1}^N \vr_{h_k} &\to\langle {\mathcal{V}^{\omega}_{t,x}}; \varrho \rangle, \ \mbox{as $N \rightarrow \infty$ a.e. in} \,\,(0,T)\times\T^3, \\
		\frac 1N \sum_{k=1}^N \vm_{h_k} &\to \langle {\mathcal{V}^{\omega}_{t,x}}; {\textbf m} \rangle, \ \mbox{as $N \rightarrow \infty$ a.e. in} \,\,(0,T)\times\T^3.
	\end{aligned}
\end{equation*}


\section{Proof of Theorem~\ref{T_ccE}: Convergence to Regular Solution}
\label{proof2}

We have proven that the numerical solutions $\{\vU_h\}_{h>0}$  to \eqref{scheme_bE} for the stochastic Euler system converges to the dissipative measure--valued martingake solution, in the sense of Definition~\ref{def:dissMartin}. Employing the corresponding weak (measure-valued)--strong uniqueness results (cf. Theorem~\ref{Weak-Strong Uniqueness_01}), we can show the strong convergence of numerical approximations to the strong solution of the system on its lifespan.

First note that, Proposition~\ref{prop:skorokhod1} and Theorem~\ref{Weak-Strong Uniqueness_01} gives the required weak-$*$ convergence. Indeed, from Proposition~\ref{prop:skorokhod1}, we have $\p$-a.s.,
$$\varrho_h(\cdot\wedge\mathfrak{t}_R)\to\langle {\mathcal{V}^{\omega}_{t,x}}; \varrho \rangle(\cdot\wedge\mathfrak{t}_R)\,\,\mbox{in}\,\,\,C_w([0,T],L^\gamma(\mathbb{T}^3)),$$
$$\mathbf{m}_h(\cdot\wedge\mathfrak{t}_R)\to\langle {\mathcal{V}^{\omega}_{t,x}}; {\textbf m} \rangle(\cdot\wedge\mathfrak{t}_R)\,\,\mbox{in}\,\,\,C_w([0,T],L^{\frac{2\gamma}{\gamma+1}}(\mathbb{T}^3)).$$ 
Combination of above convergence and Theorem~\ref{Weak-Strong Uniqueness_01} gives the required weak-* convergence. For the proof of strong convergence of density and momentum in $L^1(\T^3)$, note that from Proposition \ref{prop:skorokhod1}, Theorem \ref{Weak-Strong Uniqueness_01}, energy bounds \eqref{aprhov}-\eqref{estrhou2} and using the fact limit Young measure of any subsequence $(\delta_{\varrho_{h_k}(\cdot\wedge\mathfrak{t}_R),\mathbf{m}_{h_k}(\cdot\wedge\mathfrak{t}_R)})_{k\ge\,1}$ is $\delta_{\bar{\varrho}(\cdot \wedge \mathfrak{t}_R), \bar{\mathbf{m}}(\cdot\wedge \mathfrak{t}_R)}$, we have $\mathbb{P}$-a.s., sequence of young measure converges to dirac Young measure, i.e. $\p$-a.s.
$$ \delta_{\varrho_{h}(\cdot\wedge\mathfrak{t}_R),\mathbf{m}_{h}(\cdot\wedge\mathfrak{t}_R)}  \rightarrow \delta_{\bar{\varrho}(\cdot \wedge \mathfrak{t}_R), \bar{\mathbf{m}}(\cdot\wedge \mathfrak{t}_R)},\,\, \text{weak-$*$ in}\, \, L^{\infty}((0,T)\times \T^3; \mathcal{P}(\R^4))$$
By theory of Young measure \cite[Proposition 4.16]{Balder}, it implies that, $\p$-a.s. $\vr_{h}(\cdot\wedge\mathfrak{t}_R)$, $\vm_h(\cdot\wedge\mathfrak{t}_R)$ converges to $\bar{\varrho}(\cdot \wedge \mathfrak{t}_R), \bar{\mathbf{m}}(\cdot \wedge \mathfrak{t}_R)$ in measure respectively. Note that, $\p$-a.s. sequence $(\vr_{h}(\cdot\wedge\mathfrak{t}_R)$, $\vm_h(\cdot\wedge\mathfrak{t}_R))$ is uniformly integrable and converges in measure, therefore Vitali's convergence theorem implies that $\p$-a.s,
\begin{equation*}
	\begin{aligned}
		\vr_{h}(\cdot\wedge\mathfrak{t}_R) &\to \bar\vr(\cdot\wedge\mathfrak{t}_R) \ \mbox{ strongly in}\ L^1((0,T) \times \T^3), \\
		\vm_{h}(\cdot\wedge\mathfrak{t}_R) &\to \bar{\mathbf{m}}(\cdot \wedge \mathfrak{t}_R) \ \mbox{ strongly in}\ L^1((0,T) \times \T^3; \R^3)).
	\end{aligned}
\end{equation*}
This finishes the proof of the theorem.

\section*{Acknowledgements}
U.K. acknowledges the support of the Department of Atomic Energy,  Government of India, under project no.$12$-R$\&$D-TFR-$5.01$-$0520$, and India SERB Matrics grant MTR/$2017/000002$.


\end{document}